\documentclass{article}
\usepackage{geometry}
\usepackage{graphicx} 
\usepackage{amsthm}
\usepackage{amsfonts}
\usepackage{amsmath}
\usepackage{amsthm}
\usepackage[
    colorlinks=true,
    linkcolor=blue,
    citecolor=blue,
    urlcolor=blue
]{hyperref}
\usepackage{braket}
\usepackage{bm}
\usepackage{mathtools}
\usepackage{appendix}
\usepackage{color}
\usepackage{subcaption}
\usepackage{algorithm}
\usepackage[sort&compress]{cleveref}
\usepackage{algpseudocode} 
\usepackage{amssymb}

\usepackage{float}
\usepackage{placeins}
\usepackage{indentfirst}

\crefrangeformat{equation}{(#3#1#4--#5#2#6)}
\Crefrangeformat{equation}{(#3#1#4--#5#2#6)}
\crefmultiformat{equation}{(#2#1#3}{--#2#1#3)}{--#2#1#3}{--#2#1#3)}
\Crefmultiformat{equation}{(#2#1#3}{--#2#1#3)}{--#2#1#3}{--#2#1#3)}

\DeclareMathOperator{\Span}{span}

\theoremstyle{definition} 
\newtheorem{definition}{Definition}[section]
\theoremstyle{theorem} 
\newtheorem{theorem}{Theorem}[section]
\theoremstyle{lemma} 
\newtheorem{lemma}{Lemma}[section]
\theoremstyle{proposition} 
\newtheorem{proposition}{Proposition}[section]
\theoremstyle{corollary} 
\newtheorem{corollary}{Corollary}[section]
\theoremstyle{remark} 
\newtheorem{remark}{Remark}[section]

\newcommand{\cC}{\ensuremath{\mathcal{C}}}
\newcommand{\cD}{\ensuremath{\mathcal{D}}}

\newcommand{\cF}{\ensuremath{\mathcal{F}}}

\newcommand{\cJ}{\ensuremath{\mathcal{J}}}

\newcommand{\cL}{\ensuremath{\mathcal{L}}}

\newcommand{\cN}{\ensuremath{\mathcal{N}}}

\newcommand{\cP}{\ensuremath{\mathcal{P}}}

\newcommand{\cT}{\ensuremath{\mathcal{T}}}

\newcommand{\stiefelframe}[1]{\ensuremath{\bm{#1}}}

\newcommand{\bL}{\ensuremath{\mathbb{L}}}

\newcommand{\bP}{\ensuremath{\mathbb{P}}}

\newcommand{\bR}{\ensuremath{\mathbb{R}}}
\newcommand{\bS}{\ensuremath{\mathbb{S}}}

\newcommand{\ambientspace}{\ensuremath{L^2_{\mu_0}}}
\newcommand{\bW}{\ensuremath{\mathbb{W}}}
\newcommand{\bX}{\ensuremath{\mathbb{X}}}

\DeclareRobustCommand{\bb}{{\bm{b}}}
\DeclareRobustCommand{\bc}{{\bm{c}}}
\DeclareRobustCommand{\be}{{\bm{e}}}
\DeclareRobustCommand{\bh}{{\bm{h}}}
\DeclareRobustCommand{\bv}{{\bm{v}}}

\DeclareRobustCommand{\bh}{{{\bm{h}}}}

\newcommand{\rD}{\ensuremath{\mathrm{D}}}

\newcommand{\rH}{\ensuremath{\mathrm{H}}}

\newcommand{\rK}{\ensuremath{\mathrm{K}}}

\newcommand{\rV}{\ensuremath{\mathrm{V}}}

\newcommand{\subdifferentialelement}{\ensuremath{\partial^{\circ}\cF}}

\def\[{\left[}
  \def\]{\right]}
\def\<{\langle}
\def\>{\rangle}
\def\({\left(}
    \def\){\right)}
\def\[{\left [}
  \def\]{\right]}
\def\({\left(}
    \def\){\right)}

\newcommand{\Tan}{\ensuremath{\mathrm{Tan}}}

\newcommand{\id}{\ensuremath{\text{id}}}

\NewDocumentCommand{\Vdec}{o}{\ensuremath{V_{\varphi\IfValueT{#1}{^{#1}}}}}

\NewDocumentCommand{\Wtwometric}{o o o}{%
  \ensuremath{%
    W_2%
    \IfValueT{#3}{%
      ^{#3}%
    }%
    \IfValueT{#1}{%
      \left(#1,#2\right)%
    }%
  }%
}
\newcommand{\Tmap}[2]{\ensuremath{T_{#1\to#2}}}
\newcommand{\Toptmap}[2]{\ensuremath{T^{\text{opt}}_{#1\to#2}}}
\newcommand{\pushfwd}[2]{\ensuremath{{#1}{\#}#2}}
\newcommand{\Ptwospace}[1][]{%
  \ensuremath{\mathcal{P}_2%
    \if\relax\detokenize{#1}\relax
    \else
      \left(#1\right)%
    \fi}
}
\newcommand{\Wtwospace}[1][]{%
  \ensuremath{\mathcal{W}_2%
    \if\relax\detokenize{#1}\relax
    \else
      \left(#1\right)%
    \fi}
}
\newcommand{\Ltwoweight}[1]{\ensuremath{L^2((\bR^d,#1),\bR^d)}}

\newcommand{\Ltwow}[1]{\ensuremath{L^2_{#1}}}
\newcommand{\Leb}[1]{\ensuremath{\mathcal{L}^{#1}}}
\newcommand{\dom}{\ensuremath{\text{dom}}}

\newcommand{\dec}{\ensuremath{\mathcal{T}}}
\newcommand{\proj}[2]{\ensuremath{\mathrm{P}^{#2}_{#1}}}

\DeclarePairedDelimiter{\prt}{(}{)}

\let\set\relax
\DeclarePairedDelimiter{\set}{\{}{\}}
\DeclarePairedDelimiter{\abs}{\lvert}{\rvert}
\DeclarePairedDelimiter{\norm}{\lVert}{\rVert}
\DeclarePairedDelimiter{\inner}{\langle}{\rangle}
\DeclarePairedDelimiter{\normRd}{\lvert}{\rvert}
\DeclarePairedDelimiter{\innerRd}{\langle}{\rangle}

\newcommand{\Id}{\ensuremath{\text{Id}}}

\newcommand{\opt}{\ensuremath{\mathrm{opt}}}
\newcommand{\Lin}{\ensuremath{\mathrm{Lin}}}
\newcommand{\grad}{\ensuremath{\mathrm{grad}}}
\renewcommand{\d}{\mathrm{d}}

\newcommand{\st}{\ensuremath{\text{ s.t.~}}}
\newcommand{\vel}{\ensuremath{\text{vel}}}

\newcommand{\St}{\text{St}}

\DeclareMathOperator*{\argmin}{arg\,min}

\newcommand{\rd}{\ensuremath{\mathrm d}}
\newcommand{\dx}{\ensuremath{\mathrm dx}}
\newcommand{\ds}{\ensuremath{\mathrm ds}}
\newcommand{\dt}{\ensuremath{\mathrm dt}}
\DeclareMathOperator*{\vspan}{span}

\newcommand{\Exp}{\text{Exp}}

\newcommand{\cond}{\; :\;}

\title{A Dynamical Approximation Scheme on the Stiefel manifold for  Wasserstein Gradient Flows
}
\author{Isabella Carla Gonnella, Olga Mula, Federico Pichi, Gianluigi Rozza}
\date{}

\begin{document}
\maketitle
\begin{abstract}
  We propose a meshless Lagrangian dynamical method for approximating Wasserstein gradient flows (WGFs).
  The evolving measure is represented as the pushforward of the initial measure $\mu_0$ through a transport map in the weighted Hilbert space $L^2_{\mu_0}$.
  We approximate this map in time-dependent linear subspaces of $L^2_{\mu_0}$, whose orthonormal frames are evolved by a Dirac--Frenkel dynamical principle on a Stiefel manifold constrained to a finite-dimensional background space, adaptively constructed via local approximations of the WGF velocity field.
  We prove that the resulting transport map induces an absolutely continuous curve of probability measures in Wasserstein space, whose velocity is obtained by projecting the exact WGF velocity onto the background space, and we show that the approximation preserves the energy dissipation structure up to the projection error of the velocity.
  Moreover, for geodesically convex energies, we derive an a posteriori estimate controlling such projection error through the adaptive construction of the background space, yielding as well a bound on the approximation error of the pushforward measure in the Wasserstein metric.
  Numerical experiments on linear and nonlinear Fokker--Planck equations, porous-medium diffusion, and interaction energies demonstrate the accuracy of the method, its energy-dissipation properties, and the advantages of the adaptive construction.
\end{abstract}


\section{Introduction}
\label{sec:introduction}

Wasserstein gradient flows have become increasingly relevant for modeling a wide range of phenomena arising in physics, quantum chemistry, finance, data-driven modeling, and machine learning \cite{Otto2001,CMV2003,neklyudov2023wasserstein,conger2024coupled,alvarezmelis2020dataset,caucheteux2026unifying, de2026flowing}.
Their mathematical theory is well established, yet continues to evolve, and has led to major advances in the analysis of dissipative PDEs \cite{JKO1998, Otto2001, AGS2008}.
This paper is concerned with their efficient and accurate numerical approximation, which is still a challenging problem due to the geometric and computational complexity of the Wasserstein space.

To the best of our knowledge, two main paradigms have emerged for their numerical approximation.
The first is based on the minimizing movement scheme (JKO) \cite{JKO1998}, which approximates the evolution through a time-marching procedure requiring the solution of an optimization problem at each time step.
It includes methods based on the Benamou-Brenier dynamical formulation \cite{BB2000,benamou2016augmented}, primal-dual optimization \cite{carrillo2022primal}, entropic regularization and Sinkhorn algorithms \cite{peyre2015entropic,carlier2017convergence}, and, more recently, machine-learning implementations of JKO-type schemes \cite{mokrov2021large,lee2024deep,caucheteux2026unifying}.
The second paradigm consists of so-called Lagrangian methods.
A prominent subclass is given by particle methods, which approximate the solution by empirical averages of Dirac measures supported on moving locations \cite{craig2016blob,carrillo2019blob,carrillo2018lagrangian}.
However, more generally, Lagrangian methods represent the WGF solution as the pushforward of the initial measure through a time-dependent transport map, whose evolution is obtained by solving a PDE, typically by finite volume or finite-element schemes \cite{carrillo2010numerical,carrillo2021lagrangian}.

Each paradigm has complementary strengths and weaknesses.
On the one hand, JKO methods preserve the energy-dissipation structure at the discrete level, but require the solution of an optimal-transport problem at each time step, which rapidly becomes computationally demanding beyond one spatial dimension.
On the other hand, Lagrangian methods avoid solving such optimization problems explicitly, but typically rely on mesh-based spatial discretizations, which suffer from the curse of dimensionality already in moderate dimension.
Particle methods are a notable exception, since they are naturally meshless; however, they produce empirical measures that are singular with respect to the Lebesgue measure and therefore usually require regularization, introducing additional analytical and numerical difficulties.

The objective of this work is to develop an efficient approximation framework for WGFs.
To this end, we start from the Lagrangian representation of the exact solution, given by:
\begin{align*}
  \mu_t=\pushfwd{T_t}{\mu_0}\qquad \text{for a.e.}\quad t\geq 0,
\end{align*}
where $T_t$ is the transport map at time $t$ and $\mu_0$ is the initial measure.
Equivalently, the measure $\mu_t$ is obtained by transporting the mass of $\mu_0$ through $T_t$, in the sense that
\begin{align*}
  \int_{\bR^d} f\prt{x}\,\d\mu_t\prt{x}
  =
  \int_{\bR^d} f\circ T_t\prt{x}\,\d\mu_0\prt{x},
  \qquad \forall f\in C_b\prt{\bR^d}.
\end{align*}
Rather than relying on mesh-based discretizations commonly used in Lagrangian approaches, we approximate the transport map $T_t$ by a parametric approximation of the form
\begin{align*}
  \dec_t=\varphi\prt{\theta_t},
\end{align*}
where $\varphi:\Theta\to L^2_{\mu_0}$ is a decoder map from a finite-dimensional parameter space $\Theta$ into the $L^2_{\mu_0}$ weighted space in which transport maps are defined (see \Cref{eq:L2weighted} for the precise definition of $L^2_{\mu_0}$).
The time-evolution of the parameters is determined through a Dirac-Frenkel dynamical approximation principle, yielding an ordinary differential equation for $\theta_t$ in $\Theta$.
In this way, the original infinite-dimensional WGF is reduced to a finite-dimensional ODE system for $\theta_t$, whose time integration yields the transport map approximation $\dec_t$ and hence the approximate measure
\begin{align*}
  \nu_t=\pushfwd{\dec_t}{\mu_0}.
\end{align*}

The Dirac-Frenkel principle is a very general paradigm for constructing dynamical approximations of evolution equations.
Its origin can be traced back at least to Dirac and Frenkel's works in quantum mechanics in the 1930's \cite{dirac1930note}. The idea  has subsequently appeared in several areas under different names, including low-rank approximation, and dynamical approximation. In problems connected to optimization or gradient flows, some keywords referring to this strategy are natural (Wasserstein) gradient flows \cite{LM2018, liu2020neural, zuo2024numerical, caucheteux2026unifying}, and (energy) natural gradient \cite{amari1998natural, AD1998, AN2000, AMS2008, NLY2023}. If the decoder $\varphi$ is parameterized by a neural network, the strategy is also called neural Galerkin \cite{bruna2022neural, LN2025, bon2025stable, feischl2026regularized} (we provide additional relevant literature at the beginning of \Cref{sec:dyn-aprox}).

Its application to WGFs and closely related transport-map gradient flows is only very recent, and existing works rely exclusively on neural-network parameterizations of the decoder $\varphi$ \cite{liu2020neural,zuo2024numerical,caucheteux2026unifying,dumont2026learning}.
The Dirac--Frenkel principle determines the parameter dynamics by matching the decoder-induced velocity with the original WGF velocity in a least-squares sense, which naturally requires weighted $L^2$ inner products.
In measure-based parametrizations, the corresponding weights evolve in time, possibly causing sampling instabilities and motivating dynamical sampling strategies \cite{bon2025stable}.
In the Lagrangian formulation used here, instead, these products are pulled back to the fixed initial measure $\mu_0$, so that no adaptive sampling distribution is needed.

The present work does not represent transport maps by neural networks, but instead uses dynamically evolving low-dimensional linear subspaces of $L^2_{\mu_0}$.
These subspaces are parameterized through orthonormal bases belonging to a Stiefel manifold (see \Cref{sec:stiefel-notions} for its definition and basic properties) constrained to a finite-dimensional background space.
The construction of the background space is a key ingredient of the proposed method: it is independent of any spatial discretization and can be updated adaptively at prescribed times.

The main contributions of this work are the following:

\begin{enumerate}

  \item \textbf{Abstract formulation:}
        We develop an abstract functional framework that rigorously connects the Lagrangian formulation of WGFs with the Dirac-Frenkel dynamical approximation principle.
        This formulation is novel for WGFs, and clarifies the mathematical structure underlying the proposed method.

  \item \textbf{Stiefel decoder and Taylor-based background spaces:}
        We introduce the Stiefel decoder, whose parameters consist of $n$ coefficients and an orthonormal frame of $n$ functions constrained to an $m$-dimensional background space, with $m\geq n$.
        This formalizes the Dirac--Frenkel principle on the Stiefel manifold of orthonormal frames in a function space, and provides a continuous counterpart of mesh-based dynamical low-rank approximations, where the evolving basis is made of genuine functions rather than vectors induced by a spatial discretization.
        We then propose a Taylor-based construction of the background space, using local information from the Lagrangian velocity operator, which can be kept fixed or updated adaptively.

  \item \textbf{Error analysis:}
        We prove that the method induces absolutely continuous curves of probability measures in Wasserstein space and preserves the energy-dissipation structure up to a projection error.
        For geodesically convex energies, we show that such projection error, and consequently the measure approximation error in the Wasserstein metric, can be made arbitrarily small by suitably updating the Taylor-based background space at prescribed times.

\end{enumerate}

The remainder of the paper is organized as follows. In
\Cref{sec:wasserstein-gradient-flows}, we recall the functional framework for WGFs needed throughout the paper, introduce the class of energies under consideration, and present their Lagrangian formulation in terms of transport maps. In
\Cref{sec:lagrangian-dynamical-approximation}, we introduce the abstract dynamical approximation principle and derive the projected evolution equation governing the parameter dynamics. \Cref{sec:stiefel} specializes this framework to the proposed Stiefel decoder and derives the resulting reduced dynamical system. In \Cref{sec:tb_background}, we introduce the Taylor-based background spaces and analyze the resulting adaptive strategy through residual and final-time error estimates. Finally, \Cref{sec:numerical-results} presents numerical experiments on linear and nonlinear Fokker-Planck equations, nonlinear diffusion, and interaction energies.

\section{Wasserstein Gradient Flows (WGF)}
\label{sec:wasserstein-gradient-flows}
In this section, we recall the notion of gradient flows in the $L^2$-Wasserstein space of measures  defined on $\bR^d$, and the main theoretical notions that are needed for our numerical scheme.
Throughout the paper, we fix a final time $t_f>0$ and consider evolutions on the time interval $[0,t_f]$.

\subsection{The $L^2$-Wasserstein space}
In the following, for any $x, y\in \bR^d$, we denote by $\innerRd{x,y}$ the Euclidean scalar product, and by $\normRd{x}$ the Euclidean norm.
Moreover, $\cP\prt{\bR^d}$ is the collection of all probability measures on $\bR^d$, and
\begin{equation*}
  \Ptwospace[\bR^d] ~\coloneqq~ \set[\big]{ \mu \in \cP\prt{\bR^d} \cond \mu\prt{\bR^d} = 1, \; \int_{\bR^d} \normRd{x}^2 \,\d\mu\prt{x} \;< +\infty }\, ,
\end{equation*}
is the subset of $\cP\prt{\bR^d}$ with finite second order moments. We equip $\Ptwospace[\bR^d]$ with the Wasserstein metric which is defined as
\begin{equation}
  \label{eq:W2-static}
  \Wtwometric[\mu][\nu] ~\coloneqq~ \mathop{\min}_{\pi \in \Pi\prt{\mu,\nu}} \prt[\Bigg]{ \int_{\bR^d \times \bR^d} \normRd{x-y}^2 \,\d\pi\prt{x,y} }^{1/2},\quad \forall \prt{\mu,\nu}\in \Ptwospace[\bR^d] \times \Ptwospace[\bR^d],
\end{equation}
where
$\Pi\prt{\mu,\nu} \coloneqq \set{ \pi \in \Ptwospace[\bR^d\times\bR^d] \cond \pi\prt{\cdot \times \bR^d} = \mu, \quad\pi\prt{\bR^d \times \cdot} = \nu}$
is the set of transport plans whose marginals are $\mu$ and $\nu$ respectively. For every $\prt{\mu,\nu}$, there exists at least one optimal coupling $\pi^{\textrm{opt}}\in \Pi\prt{\mu, \nu}$ in \Cref{eq:W2-static},
which, according to Brenier's theorem \cite{Brenier1991}, is uniquely induced by a transport map $\Toptmap{\mu}{\nu}:\bR^d\to \bR^d$ when $\mu$ has Lebesgue density, in such a way that $\pi^{\textrm{opt}} = \pushfwd{\prt{\id\times \Toptmap{\mu}{\nu}}}{\mu}$.
Furthermore, \Cref{eq:W2-static} is a distance over $\Ptwospace[\bR^d]$, and $\Wtwospace[\bR^d] \coloneqq \prt{\Ptwospace[\bR^d], \Wtwometric}$ is a separable, complete metric space (see Proposition $5.1$ and $7.1.5$ in \cite{AGS2008}).

The space $\Wtwospace[\bR^d]$ is said to have a pseudo-Riemannian structure.
Concretely, this means that it is possible to define a notion of absolutely continuous curves in $\Wtwospace[\bR^d]$, whose velocity fields, for each $\mu\in \Ptwospace[\bR^d]$, belong to the tangent space
\begin{align*}
  \Tan^{\vel}_{\mu}\Ptwospace[\bR^d]
  \coloneqq \overline{\{\nabla \Phi\cond \Phi\in \cC_c^\infty\prt{\bR^d}\}}^{\Ltwow{\mu}},
\end{align*}
defined as the closure w.r.t. the Hilbert space
\begin{align}
  \label{eq:L2weighted}
  \Ltwow{\mu} \coloneqq \Ltwoweight{\mu} = \set[\big]{  v:\bR^d\to\bR^d \cond \int_{\bR^d} \normRd{v\prt{x}}^2 \d\mu\prt{x} <\infty },
\end{align}
which is endowed with the inner product
$\inner{v,w}_{\Ltwoweight{\mu}} \coloneqq \int_{\bR^d} \innerRd{v\prt{x}, w\prt{x}} \d\mu\prt{x}$.

\subsection{Gradient Flows in the Wasserstein space}

Given an ambient space $\bX$, an energy functional $\cF:\bX\to (-\infty, +\infty]$, and an initial condition $\bar{u} \in \bX$, a gradient flow is defined as a curve $\prt{u_t}_{t\in [0,t_f]}\in \bX$ that decreases the energy $\cF$ as fast as allowed by the geometry of the space $\bX$. Such a statement translates into an evolution equation of the form
\begin{align*}
  \begin{cases}
    \partial_t u_t & = - \grad_{\bX} \cF\prt{u_t}, \quad t\in [0,t_f] \\
    u_0            & = \bar{u},
  \end{cases}
\end{align*}
where $\grad_{\bX} \cF\prt{u_t}$ denotes the gradient of $\cF$ at $u_t$. The notion of gradient needs to be appropriately defined depending on the nature of the space $\bX$.
However, in all cases, it is connected to the notion of (sub)differential via an Onsager operator \cite{mielke2023introduction} that involves the metric structure of the space (see \cite{mielke2023introduction} for more details in the case of Hilbert spaces and Riemannian manifolds).

When $\bX=\Wtwospace[\bR^d]$,
the subdifferential of $\cF$ at a given $\mu\in \Ptwospace[\bR^d]$ is defined as the set
\begin{align}\label{eq:subdifferential-definition}
  \partial \cF\prt{\mu} \coloneqq \set[\Big]{\xi \in \Ltwow{\mu} \cond \liminf_{\nu\rightharpoonup  \mu}  \frac{\cF\prt{\nu} - \cF\prt{\mu} -\int_{\bR^d\times \bR^d} \prt{\xi\prt{x}, y-x} \d\pi^{\textrm{opt}}\prt{x,y}}{\Wtwometric[\mu][\nu]} \geq 0 },
\end{align}
and, if nonempty, its element of minimal $\Ltwow{\mu}$ norm is denoted by $\partial^\circ\cF\prt{\mu}$.
The Onsager operator is defined as
\begin{align*}
  \rK\prt{\mu}:\Tan^{\vel}_{\mu}\Ptwospace[\bR^d] & \to  \Tan_{\mu}\Ptwospace[\bR^d]                                                                                            \\
  v                                               & \mapsto \rK\prt{\mu}\prt{v} \coloneqq -\nabla_x \cdot\prt{\mu v} \quad\text{(in the sense of distributions $\cD'(\bR^d)$)},
\end{align*}
where the tangent space at $\mu \in \Ptwospace[\bR^d]$ is given by
\begin{align*}
  \Tan_{\mu}\Ptwospace[\bR^d] \coloneqq \set[\big]{ s\in \cD'\prt{\bR^d} \cond s =- \nabla\cdot\prt{\mu v} \text{ holds in }\cD'\prt{\bR^d} \text{ for some }v\in \Tan^{\vel}_{\mu}\Ptwospace[\bR^d]}.
\end{align*}

In full generality, $\partial^\circ\cF\prt{\mu}\in \Ltwow{\mu}$ and nothing guarantees a priori that it belongs to the smaller set  $\Tan^{\vel}_{\mu}\Ptwospace[\bR^d]$. When $\partial^\circ\cF\prt{\mu}\in \Tan^{\vel}_{\mu}\Ptwospace[\bR^d]$, we can apply the Onsager operator and obtain a notion of Wasserstein gradient at $\mu$ as
\begin{align*}
  \grad_{\Wtwometric} \cF\prt{\mu} \coloneqq \rK\prt{\mu}\prt{\partial^\circ\cF\prt{\mu}}.
\end{align*}
To ensure that this conditions holds through our developments, in what follows we focus on energies of the form
\begin{equation}
  \label{eq:main-energy-functional}
  \cF\prt{\mu} \coloneqq
  \begin{cases}
    \int_{\bR^d} F\prt{x, \rho, \nabla \rho}\dx & \text{ if } \mu \in \dom\prt{\cF} \coloneqq \set{\rho \Leb{d} \cond \rho \in \cC^1\prt{\bR^d}} \\
    +\infty                                     & \text{ otherwise},                                                                             \\
  \end{cases}
\end{equation}
where $\Leb{d}$ denotes the Lebesgue measure on $\bR^d$, and
$F:\bR^d\times\bR_+\times\bR^d\to\bR_+$ is at least of class $\cC^2$ and
satisfies $F\prt{x,0,p}=0$ for every $\prt{x,p}\in\bR^d\times\bR^d$, so that vacuum
regions do not contribute to the energy.
For this class of functionals it holds that
\begin{equation*}
  \partial^\circ\cF\prt{\mu} \coloneqq
  \begin{cases}
    \nabla_x \prt[\Big]{\frac{\delta\cF}{\delta\mu}\prt{\mu}}, & \quad \text{ if } \mu \in \dom\prt{\cF} \\
    \emptyset,                                                 & \quad \text{otherwise,}
  \end{cases}
\end{equation*}
where the function $\frac{\delta\cF}{\delta\mu}\prt{\mu}$ is the first variation of $\cF$ at $\mu$, given by (see Definition 7.12 of \cite{Santambrogio2015})
$$
  \frac{\delta\cF}{\delta\mu}\prt{\mu}\prt{x} = F_z\prt{x, \rho, \nabla \rho} - \nabla_x \cdot F_p\prt{x, \rho, \nabla \rho},\quad \forall x\in \bR^d \text{ and with $\mu=\rho \Leb{d}$},
$$
and $F_z$ and $F_p$ denote the partial derivates of $F$ with respect to the second and third variable respectively.
Within this setting, by construction, $\partial^\circ\cF\prt{\mu}\in\Tan_\mu^{\vel}\Ptwospace[\bR^d]$ under sufficient regularity assumptions on $F$ and $\rho$.
Therefore, given an initial condition
$\bar{\mu}=\bar{\rho}\Leb{d}\in\dom\prt{\cF}$ and an energy functional of the form
\Cref{eq:main-energy-functional}, we say that a curve
$\prt{\mu_t}_{t\in[0,t_f]}$, with $\mu_t=\rho_t\Leb{d}$, solves the WGF if it
satisfies
\begin{equation}
  \label{eq:cont-eq-wgf}
  \begin{cases}
    \partial_t \mu_t & = \nabla_x \cdot\prt[\Big]{\mu_t \nabla_x \prt[\Big]{\frac{\delta\cF}{\delta\mu}\prt{\mu_t}}}, \quad \forall t\in[0,t_f], \\
    \mu_0            & = \bar{\mu},
  \end{cases}
\end{equation}
namely the continuity equation with velocity field
$v_{\mu_t}\prt{x}\coloneqq-\partial^\circ\cF\prt{\mu_t}\prt{x}
  =-\nabla_x\prt[\Big]{\frac{\delta\cF}{\delta\mu}\prt{\mu_t}}\prt{x}$, understood in the sense of distributions:
\begin{equation}
  \label{eq:continuity-eq-weak}
  \int_{\bR_+\times \bR^d} \prt[\big]{ \partial_t \varphi\prt{x,t}+v_{\mu_t}\prt{x}\cdot \nabla_x \varphi_{\mu_t}\prt{x} } \d\mu_t\prt{x}\dt =0,
  \quad \forall \varphi\in \cC_c^\infty\prt{\bR_+\times \bR^d}.
\end{equation}

Finally, in the setting of \Cref{eq:main-energy-functional}, with sufficient regularity assumptions on $F$ and $\rho$ (see Proposition $8.1.8$ in \cite{AGS2008}), the WGF admits a Lagrangian representation, i.e.\ there exists a curve of transport maps $\prt{T_t}_{t\in[0,t_f]}$, such that
$\mu_t = \pushfwd{T_t}{\mu_0}$, which is the unique solution to the \textit{Lagrangian system}
\begin{equation}
  \label{eq:push-fwd-evolution-all-regular}
  \begin{cases}
    \partial_t T_t\prt{x} & = v_{\mu_t}\circ T_t\prt{x}, \quad \prt{t, x}\in (0, t_f]\times \bR^d, \\
    T_0\prt{x}            & = x, \quad \forall x\in \bR^d.
  \end{cases}
\end{equation}
We underline that, since $v_{\mu_t}=-\partial^\circ\cF\prt{\mu_t}\in\Ltwow{\mu_t}$, the push-forward formula yields
$\partial_t T_t\in\Ltwow{\mu_0}$ for a.e.\ $t\in[0,t_f]$. Consequently, $T_t \in L^2_{\mu_0}$ for a.e.\ $t\in[0,t_f]$.
In the following, we use the shorthand notation
$\dot T_t$ to indicate $\partial_t T_t$.

\paragraph{Prototypical family of WGFs.} A particular case of \Cref{eq:main-energy-functional}, arising in numerous physical and biological applications, including models in granular media, material science, biology, and machine learning applications is given by energies of the form
\begin{equation}\label{eq:prototypical-energy}
  \cF\prt{\mu}
  =
  \begin{cases}
    \displaystyle
    \int_{\bR^d} U\prt{\rho\prt{x}}\,\dx
    +\int_{\bR^d} V\prt{x}\rho\prt{x}\,\dx
    +\frac12\int_{\bR^d} \prt{W* \rho}\prt{x}\rho\prt{x}\,\dx,
     & \text{if } \mu\in \dom\prt{\cF}, \\
    +\infty,
     & \text{otherwise,}
  \end{cases}
\end{equation}
with internal energy $U:[0,+\infty)\to(-\infty,+\infty]$, internal potential $V:\bR^d\to(-\infty,+\infty]$, and
interaction potential $W:\bR^d\to(-\infty,+\infty]$.
Under suitable convexity assumptions on $U$, $V$, and $W$
(see Section~9.3 in \cite{AGS2008}), such energies are $\lambda$-convex along
generalized geodesics. This property is central to our analysis, since it
yields stability estimates for the corresponding WGFs
(see \Cref{app:wasserstein-convexity}). In the rest of the paper, we focus on
examples belonging to this class, including linear and nonlinear
Fokker--Planck equations, porous-medium equations, and some interaction systems.
\section{A Lagrangian Dynamical Approximation Framework}
\label{sec:lagrangian-dynamical-approximation}
In this work we adopt a Lagrangian point of view for approximating WGFs: we replace the curve of transport maps $\prt{T_t}_{t\in[0,t_f]}$, which solves
\Cref{eq:push-fwd-evolution-all-regular}, by a curve of approximation maps
$\prt{\dec_t}_{t\in[0,t_f]}$. We then approximate the exact solution
$\mu_t=\pushfwd{T_t}{\mu_0}$ by the pushforward measure
$\nu_t \coloneqq \pushfwd{\dec_t}{\nu_0}$.
In our method, the initial measure is represented exactly, so that $\nu_0=\mu_0$.
However, this is not the case for all Lagrangian approaches.
In this section we first review the main Lagrangian approximation methods that have been proposed in the literature, and then introduce the adopted dynamical approximation framework.

\subsection{Lagrangian approximation methods}\label{subsec:lagrangian_methods}
Lagrangian approximation strategies have been widely explored in the literature.
Two main approaches are the following ones:
\begin{itemize}
  \item \textbf{Particle methods:} One fixes $N$ points $\{y_i\}_{i=1}^N$ in $\bR^d$ and replaces the initial
        measure $\mu_0$ by the empirical measure
        $\nu_0=\frac1N\sum_{i=1}^N\delta_{y_i}$.
        The Lagrangian evolution then transports each atom into a moving particle:
        if $x_i\prt{t}=T_t\prt{y_i}$, then
        $\pushfwd{T_t}{\delta_{y_i}}=\delta_{x_i\prt{t}}$. Hence one looks for a measure of
        the form $\nu_t=\frac1N\sum_{i=1}^N\delta_{x_i\prt{t}}$, where the particle locations are formally evolved according to the velocity $v_{\nu_t}\coloneqq-\subdifferentialelement\prt{\nu_t}$:
        \begin{align*}
          \begin{cases}
            \partial_t x_i\prt{t}
            =
            v_{\nu_t}\prt{x_i\prt{t}}, \\
            x_i\prt{0}=y_i,
          \end{cases}
          \qquad i=1,\dots,N.
        \end{align*}
        Its main limitation is that the approximation of $\mu_t$ is only available at the particle locations.
        Moreover, the empirical measure $\nu_t$ is singular with respect to the Lebesgue measure.
        For the class of energies considered in \Cref{eq:main-energy-functional}, this means that $\nu_t \notin \dom\prt{\cF}$ and therefore $\partial \cF\prt{\nu_t}=\emptyset$.
        Hence, the above particle dynamics is only formal, unless one introduces a regularization or mollification of $\nu_t$, which raises additional analytical and numerical difficulties.

  \item \textbf{Discretization on spatial meshes:} Another class of Lagrangian methods is based on finite element or finite volume discretizations of the transport map on a spatial mesh.
        In this setting, one typically approximates $T_t$ at each time $t\in[0,t_f]$ by piecewise polynomials, and evolves its degrees of freedom in time.
        The spatial mesh can be kept fixed, but there are also moving-mesh approaches in which the grid nodes are transported along the Lagrangian flow.
        Thus, the spatial discretization adapts to the possible concentration of the mass, and the density is reconstructed by enforcing mass conservation on each cell \cite{carrillo2018lagrangian,carrillo2021lagrangian}.
        The main drawback of mesh-based methods is that dealing with spatial grids becomes challenging already in moderate dimensions.
\end{itemize}
In this work, we propose a dynamical approximation framework that addresses
these limitations: it avoids spatial meshes, which is essential for reducing
the dimensional bottleneck of grid-based methods, while producing measure
approximations that can be evaluated throughout the whole domain, so that
densities, energies, and velocity fields remain accessible beyond a finite set
of particle locations.

\subsection{Dynamical approximation of the transport map}
\label{sec:dyn-aprox}
Dynamical approximation refers to a broad class of strategies introduced by Dirac and Frenkel \cite{dirac1930note,frenkel1934wave} and has attracted significant
interest across several fields. Given that the idea is very general, it has recently been (re)discovered and presented by different communities with different names, such as Dirac–Frenkel approximation (in physics and chemistry \cite{mclachlan1964variational,lubich2008quantum}), parametric approximations
(in scientific computing \cite{lee2020model,peherstorfer2020model}), dynamical low rank (for matrix approximation \cite{sapsis2009dynamically,koch2007dynamical}), neural Galerkin (for neural network
approximation \cite{bruna2022neural}), and natural gradient flows (for optimization and machine learning tasks \cite{amari1998natural,nurbekyan2023efficient}).
All these viewpoints are based on the same principle: the solution of the original system is
approximated by a decoder function which belongs to a parametrized set of functions. The parameters are
evolved so that the induced velocity best matches the velocity of the exact dynamics.

As anticipated in the introduction, to build a Lagrangian numerical scheme for WGFs we approximate $T_t$ with
\begin{align}\label{eq:ansatz}
  \dec_t\coloneqq \varphi\prt{\theta_t}, \quad \forall t\in [0, t_f],
\end{align}
where $\varphi:\Theta\to L^2_{\mu_0}$ is a decoder mapping from a finite-dimensional parameter space $\Theta$ into the function space $L^2_{\mu_0}$ in which transport maps are defined.
In full generality, $\Theta$ can be chosen as a finite-dimensional Riemannian manifold, although the most common choice is the flat space $\Theta=\bR^n$.

Well known approximation classes including finite elements, radial bases or neural networks can be written in the general language of decoders by choosing the appropriate parameter $\Theta$ and the appropriate definition of the mapping $\varphi$.
The image of the decoder is the parametized subset of the ambient space denoted by
\begin{align*}
  \Vdec \coloneqq \varphi\prt{\Theta}=\{ \varphi\prt{\theta} \cond \theta \in \Theta \} \subset \ambientspace,
\end{align*}
and represents the admissible approximations.
We assume that $\varphi$ is differentiable, so that its differential at a point $\theta\in \Theta$ is the linear mapping $\prt{D\varphi}_\theta \in \Lin\prt{\Tan_\theta \Theta, \ambientspace}$ which maps elements from the tangent space $\Tan_\theta \Theta$ to the ambient space $\ambientspace$.
We write the image set of this mapping as
$
  \Tan_{\varphi\prt{\theta}} \Vdec \coloneqq \prt{D\varphi}_\theta\prt{\Tan_\theta\Theta},
$
and define the adjoint $\prt{D\varphi}_\theta^*: \Tan_{\varphi\prt{\theta}}\Vdec \to \Tan_\theta \Theta$ as
\begin{equation}
  \label{eq:adjoint}
  \inner{ \prt{D\varphi}_\theta\prt{\dot \theta}, \dot v }_{\ambientspace}
  =
  \inner{ \dot \theta, \prt{D\varphi}_\theta^*\prt{\dot v} }_{\Tan_\theta \Theta},
  \quad \forall  \prt{\dot \theta, \dot v} \in \Tan_\theta \Theta \times \Tan_{\varphi\prt{\theta}} \Vdec.
\end{equation}


We now explain how to evolve $\dec_t$ by the Dirac-Frenkel principle. At $t=0$, we choose $\theta_0\in\Theta$ such that $\dec_0 = \varphi\prt{\theta_0}=\id$. This way, we perfectly match the initial condition with
$$
  \nu_0 \coloneqq \pushfwd{\dec_0}{\mu_0} = \mu_0.
$$
For later times $t>0$, the Dirac-Fenkel strategy consists in choosing $\dot{\dec}_t$ as the best approximation in $\Tan_{\varphi\prt{\theta_t}}\Vdec$ of $\dot{T}_t$. Assuming that $\prt{\theta_t}_{t\in[0,t_f]}\in \cC^1\prt{[0,t_f],\Theta}$, the velocity induced by the parametrization is
\begin{equation}
  \dot{\dec}_t
  =
  \frac{\d}{\dt}\varphi\prt{\theta_t}
  =
  \prt{D\varphi}_{\theta_t}\prt{\dot{\theta}_t}
  \in \Tan_{\varphi\prt{\theta_t}}\Vdec,
\end{equation}
and we choose $\dot{\dec}_t$ as the minimizer of the residual with respect to the Wasserstein velocity $v_{\nu_t}\coloneqq-\subdifferentialelement\prt{\nu_t}$:
\begin{equation}
  \label{eq:dirac-frenkel}
  \dot{\dec}_t
  =
  \prt{D\varphi}_{\theta_t}\prt{\dot{\theta}_t}
  =
  \argmin_{\dot v\in \Tan_{\varphi\prt{\theta_t}}\Vdec}
  \norm{\dot v-v_{\nu_t}\circ\dec_t}^2_{\ambientspace},
  \quad \forall t\in[0,t_f].
\end{equation}
Since we can write any $\dot v\in \Tan_{\varphi\prt{\theta_t}}\Vdec$ as $\prt{D\varphi}_{\theta_t}\prt{\dot \eta}$ for $\dot \eta \in \Tan_{\theta_t} \Theta$, we can alternatively express \Cref{eq:dirac-frenkel} as a minimization problem over the velocity in the parameter space as
\begin{equation}
  \label{eq:dirac-frenkel-params}
  \dot{\theta}_t
  =
  \argmin_{\dot \eta \in \Tan_{\theta_t}\Theta}
  \norm{\prt{D\varphi}_{\theta_t}\prt{\dot \eta}-v_{\nu_t}\circ\dec_t}^2_{\ambientspace},
  \quad \forall t\in[0,t_f].
\end{equation}
We may observe that \Cref{eq:dirac-frenkel} is a linear least-squares problem whose unique solution is
$$
  \dot{\dec}_t=\proj{\Tan_{\varphi\prt{\theta_t}}\Vdec}{\ambientspace}\prt{v_{\nu_t}\circ\dec_t},
$$
where $\proj{\Tan_{\varphi\prt{\theta_t}}\Vdec}{\ambientspace}:L^2_{\mu_0} \to \Tan_{\varphi\prt{\theta_t}} \Vdec$ denotes the orthogonal projection onto $\Tan_{\varphi\prt{\theta_t}} \Vdec$. The necessary optimality conditions read
\begin{equation}
  \label{equ:gradient_flow_Hilbert_DF_tested}
  \inner{\prt{D\varphi}_{\theta_t}\prt{\dot{\theta}_t}, \dot{v}}_{\ambientspace} = \inner{v_{\nu_t}\circ\dec_t, \dot{v}}_{\ambientspace}, \quad \forall \dot{v}\in \Tan_{\varphi\prt{\theta_t}}\Vdec.
\end{equation}
If $\Theta$ is a finite-dimensional Riemannian manifold with $\dim\prt{\Theta}=n$, then by taking a basis $\{e_i\}_{i=1}^n$, we get $\Tan_{\varphi\prt{\theta_t}}\Vdec = \vspan\set{\prt{D\varphi}_{\theta_t}\prt{e_1}, \dots, \prt{D\varphi}_{\theta_t}\prt{e_n}}$, and \Cref{equ:gradient_flow_Hilbert_DF_tested} is equivalent to
\begin{equation*}
  \inner[\big]{ \prt{D\varphi}_{\theta_t}^* \circ \prt{D\varphi}_{\theta_t}\prt{\dot{\theta}_t}, e_i }_{\Tan_{\theta_t}\Theta} = \inner{\prt{D\varphi}_{\theta_t}^* \prt{v_{\nu_t}\circ\dec_t}, e_i }_{\Tan_{\theta_t}\Theta}, \quad \forall i=1, \cdots, n.
\end{equation*}
Therefore, we obtain the following ODE evolution in the finite-dimensional parameter space $\Theta$:
\begin{equation}
  \label{equ:gradient_flow_Hilbert_normal}
  \prt{D\varphi}_{\theta_t}^* \circ \prt{D\varphi}_{\theta_t}\prt{\dot{\theta}_t} = \prt{D\varphi}_{\theta_t}^*\prt{v_{\nu_t}\circ\dec_t}, \quad \forall t\in[0, t_f].
\end{equation}

\section{Stiefel Decoder Dynamical Approximation}\label{sec:stiefel} 
In this section, we specialize the abstract Dirac--Frenkel framework introduced
above to the decoder used in this work, which we call the Stiefel decoder.
In \Cref{sec:stiefel-dyn-scheme}, we define it, describe its
parameter space, and derive the corresponding projected dynamical system.
Then, in \Cref{sec:stiefel-approx-properties}, we study the main approximation
properties of the induced curve of measures, including absolute continuity,
energy dissipation, and Wasserstein error estimates.




\subsection{Formulation of the dynamical scheme}\label{sec:stiefel-dyn-scheme}
In our work we focus on a decoder $\varphi$ based on the Stiefel manifold of $n$ orthonormal frames in the $L^2_{\mu_0}$ space, defined as
\begin{equation}
  \St\prt{n,L^2_{\mu_0}}
  \coloneqq
  \{
  \stiefelframe{v}=\{v_1,\dots, v_n\} \in \prt{\ambientspace}^n \cond \inner{v_i, v_j}_{L^2_{\mu_0}}=\delta_{i,j},\;\; 1\leq i,j\leq n
  \}.
\end{equation}
In \Cref{sec:stiefel-notions}, we recall the main concepts related to
its manifold structure: the tangent space containing admissible
infinitesimal variations of the frame, its horizontal--vertical decomposition
separating genuine changes of the represented subspace from rotations
inside the frame, the metric tensor $g_{\stiefelframe{v}}$ defining inner products
between tangent directions, and the exponential map $\Exp_{\stiefelframe{v}}$ moving frames along the
manifold while preserving orthonormality.
In practice, to work with $\St\prt{n,L^2_{\mu_0}}$ we need to restrict the frames to belong to a finite-dimensional, background subspace $\bW_m\subset L^2_{\mu_0}$ of dimension $m\geq n$, yielding the submanifold
\begin{equation}
  \St\prt{n,\bW_m; L^2_{\mu_0}}
  \coloneqq
  \{
  \stiefelframe{v}=\{v_1,\dots, v_n\} \in \bW_m \cond \inner{v_i, v_j}_{L^2_{\mu_0}}=\delta_{i,j},\;\; 1\leq i,j\leq n
  \}.
\end{equation}
Given a background space $\bW_m$, we define the Stiefel decoder as
\begin{align}
  \label{eq:decoder-stiefel-background-W}
  \varphi : \Theta                     & \to L^2_{\mu_0}                                     \\
  \theta = \prt{\bc, \stiefelframe{v}} & \mapsto \varphi\prt{\theta} = \sum_{i=1}^n c_i v_i.
\end{align}
Its image decoder is $\Vdec = \bW_m$, and the parameter space $\Theta \coloneqq \bR^n \times \St\prt{n,\bW_m;L^2_{\mu_0}}$ is itself a Riemannian manifold.
In particular, its tangent space at $\theta\in \Theta$ is $\Tan_\theta \Theta = \bR^n \times \Tan_{\stiefelframe{v}} \St\prt{n,\bW_m;L^2_{\mu_0}}$, which we endow with the metric tensor
\begin{align*}
  g_\theta\prt{\prt{\bh, \stiefelframe{w}}, \prt{\tilde \bh, \tilde{\stiefelframe{w}}}} \coloneqq \prt{\bh,\tilde \bh}+ g_{\stiefelframe{v}}\prt{\stiefelframe{w}, \tilde{\stiefelframe{w}}}, \quad \forall \prt{\bh, \stiefelframe{w}}, \prt{\tilde \bh, \tilde{\stiefelframe{w}}} \in \Tan_\theta \Theta,
\end{align*}
while its exponential map at $\theta=\prt{\bc, \stiefelframe{v}}$ is defined as the mapping
\begin{align*}
  \Exp_\theta: \Tan_\theta \Theta & \to \Theta                                                              \\
  \prt{\bh, \stiefelframe{w}}     & \mapsto \prt{\bc+\bh,\, \Exp_{\stiefelframe{v}}\prt{\stiefelframe{w}}}.
\end{align*}
The decoder $\varphi$ is differentiable, and its differential at a point $\theta=\prt{\bc, \stiefelframe{v}}\in \Theta$ is the linear mapping $\prt{\rD \varphi}_\theta \in \Lin\prt{\Tan_\theta \Theta, L^2_{\mu_0}}$ given by
\begin{align*}
  \prt{\rD \varphi}_\theta \prt{\prt{\bh, \stiefelframe{w}}} \coloneqq \sum_{i=1}^n \prt{c_i w_i + h_i v_i}, \quad \forall \prt{\bh, \stiefelframe{w}} \in \Tan_\theta \Theta.
\end{align*}
Finally, the decoder's tangent space is characterized as
\begin{equation}\label{eq:dec_tan_space}
  \Tan_{\varphi\prt{\theta}}\Vdec
  =
  \begin{cases}
    \vspan\set{\stiefelframe{v}}, & \text{if } \bc=0,     \\
    \bW_m,                        & \text{if } \bc\neq 0.
  \end{cases}
\end{equation}
For this particular form of the decoder, the dynamical approximation formula in \Cref{eq:dirac-frenkel-params} reads as
\begin{equation}
  \label{eq:least-squares_g}
  \dot \theta_t = \prt{\dot{\bc}_t, \dot{\stiefelframe{v}}_t}
  \;\in\;\;
  \argmin_{\substack{\prt{\bh, \stiefelframe{w}} \,\in\, \\\bR^n\times \Tan_{\stiefelframe{v}_t}\St\prt{n,\bW_m;L^2_{\mu_0}}}}\;
  \frac 1 2
  \norm[\Big]{\sum_{i=1}^n h_i v_{t,i} + c_{t,i} w_i -v_{\nu_t} \circ \dec_t}_{L^2_{\mu_0}}^2.
\end{equation}
In the following \Cref{thm:params-evolution}, we show that \Cref{eq:least-squares_g} leads to an explicit ODE that describes the evolution of the parameters in $\Theta$ for every $m\geq n$.
In fact, once this is proved, the parameters can be evolved on $\Theta$ by an explicit time-integration scheme on the manifold.
In particular, given a time step $\Delta t>0$ and the current parameter
$\theta_t=\prt{\bc_t,\stiefelframe{v}_t}$, let
$\dot\theta_t=\prt{\dot\bc_t,\dot{\stiefelframe{v}}_t}$ be a minimizer
of \Cref{eq:least-squares_g}. An explicit Euler step on $\Theta$
is then defined by
\begin{align*}
  \theta_{t+\Delta t}
   & \coloneqq
  \Exp_{\theta_t}\prt{\Delta t\,\dot\theta_t} \\
   & =
  \prt{
    \bc_t+\Delta t\,\dot\bc_t,\,
    \Exp_{\stiefelframe{v}_t}
    \prt{\Delta t\,\dot{\stiefelframe{v}}_t}
  }.
\end{align*}
Equivalently, writing
$\theta_{t+\Delta t}=\prt{\bc_{t+\Delta t},\stiefelframe{v}_{t+\Delta t}}$, the
updated transport map is
\begin{align*}
  \dec_{t+\Delta t}
  \coloneqq
  \varphi\prt{\theta_{t+\Delta t}}
  =
  \sum_{i=1}^n c_{t+\Delta t,i} v_{t+\Delta t,i}.
\end{align*}




\begin{theorem}
  \label{thm:params-evolution}
  For a.e.~$t>0$, the unique solution to \Cref{eq:dirac-frenkel} is
  \begin{align*}
    \dot{\dec}_t = \sum_{i=1}^n c_{t, i} \dot{v}_{t,i} +  \dot{c}_{t, i} v_{t,i}=\proj{\Tan_{\varphi\prt{\theta_t}}\Vdec}{\ambientspace}\prt{v_{\nu_t}\circ\dec_t} ,
  \end{align*}
  and the parameters $\prt{\bc_t, \bv_t}$ evolve following an ODE given by \Cref{eq:least-squares_g}, which admits minimizers satisfying for each $i\in\{1,\dots,n\}$:
  \begin{equation}
    \label{eq:params-evolution}
    \begin{cases}
      \dot{c}_{t, i}
       & =
      \inner{v_{\nu_t} \circ \dec_t, v_{t,i}}_{L^2_{\mu_0}}, \\
      \dot{v}_{t,i}
       & =
      \sum_{j=1}^{k} a_{t,ij} \psi_{t,j},
    \end{cases}
  \end{equation}
  where $k=m-n$, $\{\psi_{t,j}\}_{j=1}^{k}$ is an
  $L^2_{\mu_0}$-orthonormal basis of
  $\bL_{\stiefelframe{v}_t}\coloneqq\bW_m \cap \vspan\set{\stiefelframe{v}_t}^{\perp_{L^2_{\mu_0}}}$, and $A_t=\{a_{t,ij}\}_{i,j=1}^{n,k}\in\bR^{n\times k}$ is defined as follows:
  \begin{equation}
    \label{eq:A-star}
    A_t =
    \begin{cases}
      \text{any }A \in \bR^{n\times k},                                                                          & \text{if }\bc_t=0,     \\
      \frac{1}{\abs{\bc_t}^2} \bc_t \bb_t^\top + Z_t, \text{ with }Z_t \in \bR^{n\times k}\st Z_t^\top \bc_t =0, & \text{if }\bc_t\neq 0,
    \end{cases}
  \end{equation}
  with $\bb_t
    \coloneqq
    \{
    \inner{\proj{\bL_{\stiefelframe{v}_t}}{L^2_{\mu_0}}
    \prt{v_{\nu_t}\circ \dec_t},\psi_{t,j}}_{L^2_{\mu_0}}
    \}_{j=1}^k
    \in \bR^k$.
\end{theorem}

\begin{proof}
  We present the full proof in \Cref{sec:params-evolution} to avoid disrupting the main flow of the text.
\end{proof}

\begin{remark}\label{rmk:minimizer_rank1}
  The minimizers characterized in \Cref{thm:params-evolution} are not unique.
  In particular, since the minimizers are constructed in the horizontal subspace of the Stiefel tangent space, one can choose them with pairwise orthogonal components, so that the closed formula for the Stiefel exponential reported in \Cref{eq:exp_closed_formula} of \Cref{sec:stiefel-notions} is directly applicable.
  To this aim, assume that $\bc_t\neq 0$ and choose
  $i_0\in\{1,\dots,n\}$ such that $c_{t,i_0}\neq 0$, then the rank-one matrix
  \begin{align*}
    A_t
    =
    \frac{1}{c_{t,i_0}}\be_{i_0}\bb_t^\top
  \end{align*}
  satisfies $A_t^\top\bc_t=\bb_t$, thus being optimal.
  Moreover,
  $A_tA_t^\top
    =
    \frac{\abs{\bb_t}^2}{c_{t,i_0}^2}\be_{i_0}\be_{i_0}^\top$
  is diagonal, so that the corresponding horizontal frame has pairwise orthogonal components.
\end{remark}

\begin{remark}
  A direct consequence of \Cref{thm:params-evolution} is that the approximate transport map is constrained to the background space, namely $\dec_t\in\bW_m$ for all $t\in[0,t_f]$.
  Moreover, the ODEs in \Cref{eq:params-evolution} realize the projection of the pulled-back Wasserstein velocity onto the decoder tangent space: the coefficients $\dot c_{t,i}$ select the projection along $V_t\coloneqq\vspan\set{\stiefelframe{v}_t}$, while the frame velocities $\dot v_{t,i}$, through the matrix $A_t$, select the projection on $\bL_{\stiefelframe{v}_t}=\bW_m\cap\vspan\set{\stiefelframe{v}_t}^{\perp_{L^2_{\mu_0}}}$.
  Summing these two orthogonal contributions and recalling \Cref{eq:dec_tan_space} yields
  \begin{align}\label{eq:projected_dyn}
    \dot\dec_t =
    \begin{cases}
      \proj{V_t}{L^2_{\mu_0}}
      \prt{v_{\nu_t}\circ\dec_t},
       & \text{if } \bc_t=0,      \\[0.3em]
      \proj{\bW_m}{L^2_{\mu_0}}
      \prt{v_{\nu_t}\circ\dec_t},
       & \text{if } \bc_t\neq 0 .
    \end{cases}
  \end{align}
\end{remark}

\subsection{Continuity, dissipation, and error control}\label{sec:stiefel-approx-properties}
In this section, we establish the main properties of the dynamical approximation scheme introduced in \Cref{sec:stiefel-dyn-scheme} for WGFs.
First, in \Cref{proposition:proj_continuity_eq}, we show that an absolutely continuous curve of approximate transport maps in $L^2_{\mu_0}$ induces an absolutely continuous curve of pushforward measures in $\Wtwospace[\bR^d]$.
Then, in \Cref{thm:gronwall}, we derive a Gronwall-type estimate for the Wasserstein error between the exact WGF solution and its projected approximation. The error bound involves computable quantities and can be used as an a posteriori error estimator.
Finally, in \Cref{prop:projected-energy-decay}, we prove that the approximation preserves energy dissipation up to the tangent-space projection error.
The assumptions on the approximate transport-maps required by these results, e.g.\ absolute continuity of the curve of maps, are verified in \Cref{sec:tb_background} for the specific Taylor-based construction of the background space.

\begin{proposition}[Absolute continuity induced by the pushforward]
  \label{proposition:proj_continuity_eq}
  Let $\nu_0\in\Ptwospace[\bR^d]$, and let $\prt{\dec_t}_{t\in[0,t_f]} \in AC\prt{[0,t_f],L^2_{\nu_0}}$ be a curve of diffeomorphisms solving \Cref{eq:projected_dyn} such that $\nu_t=\pushfwd{\dec_t}{\nu_0}$.
  Then $\nu_t\in AC\prt{[0,t_f],\Ptwospace[\bR^d]}$.
\end{proposition}

\begin{proof}
  To keep the notation compact, we write the projected dynamics in terms of a time-dependent closed linear subspace $\bS_t\subset L^2_{\nu_0}$.
  Depending on the regime in \Cref{eq:projected_dyn}, this subspace is either the background space $\bW_m$ or the current frame space $V_{t,n}\coloneqq\vspan\set{\stiefelframe{v}_t}$.
  Moreover, by the fact that the pullback operator maps $\bS_t^v\coloneqq
    \set{w\circ\dec_t^{-1}\cond w\in\bS_t}$ onto $\bS_t$ and it is unitary, we get that
  \begin{align}\label{eq:projection-pullback}
    \proj{\bS_t}{L^2_{\nu_0}}\prt{z\circ\dec_t}
    =
    \proj{\bS_t^v}{L^2_{\nu_t}}\prt{z}\circ\dec_t,
    \qquad
    \forall z\in L^2_{\nu_t}.
  \end{align}
  Let $\phi\in C_c^\infty\prt{\bR^d}$, and define
  \begin{align*}
    F_t:=\int_{\bR^d} \phi\prt{x}\,\d\nu_t\prt{x}=\int_{\bR^d} \phi\circ\dec_t\prt{x}\,\d\nu_0\prt{x}.
  \end{align*}
  Since $\prt{\dec_t}_{t\in[0,t_f]}\in AC\prt{[0,t_f];L^2_{\nu_0}}$ by assumption, the chain rule gives, for $\mathcal L^1$-a.e.\ $t\in\prt{0,t_f}$,
  \begin{align*}
    \frac{\d}{\dt} F_t
     & =
    \int_{\bR^d} \inner{\nabla\phi \circ \dec_t,\, \dot{\dec}_t} \d\nu_0\prt{x}                                                                                                 \\
     & =
    \int_{\bR^d} \inner[\big]{\nabla\phi \circ \dec_t,\, \proj{\bS_t}{L^2_{\nu_0}}\prt{v_{\nu_t}\circ \dec_t}} \d\nu_0\prt{x}                                                   \\
     & =
    \int_{\bR^d} \inner[\big]{\nabla\phi \circ \dec_t,\, \proj{\bS_t^v}{L^2_{\nu_t}}\prt{v_{\nu_t}}\circ \dec_t} \d\nu_0\prt{x} \quad \text{(by \Cref{eq:projection-pullback})} \\
     & =
    \int_{\bR^d} \inner{\nabla\phi ,\, \tilde{v}_{\nu_t}} \d\nu_t\prt{x}.
  \end{align*}
  Multiplying by a smooth compactly supported function of time and integrating over $\prt{0,t_f}$ gives the weak continuity equation for separated test functions, and the extension to arbitrary $\varphi\in C_c^\infty\prt{\prt{0,t_f}\times\bR^d}$ follows by density.
  Finally, since \Cref{eq:projected_dyn} holds, for
  $\mathcal L^1$-a.e.\ $t\in[0,t_f]$ we get
  \begin{align*}
    \tilde{v}_{\nu_t}\circ\dec_t
    =
    \proj{\bS_t}{L^2_{\nu_0}}\prt[\Big]{
      v_{\nu_t}\circ\dec_t
    }
    =
    \dot{\dec}_t .
  \end{align*}
  Morover, given by assumption that $\dec\in AC\prt{[0,t_f];L^2_{\nu_0}}$, it follows that the mapping
  $t\mapsto\norm{\dot{\dec}_t}_{L^2_{\nu_0}}$ belongs to $L^1\prt{0,t_f}$.
  Hence
  $\prt{\tilde{v}_{\nu_t}}_{t\in[0,t_f]}\in L^1\prt{[0,t_f],L^2_{\nu_t}}$, so that, by Theorems 8.3.1 and 8.5.1 in \cite{AGS2008}, the weak continuity equation obtained above defines an absolutely continuous curve in $\Wtwospace[\bR^d]$.
\end{proof}

\begin{theorem}[A Gronwald-type a posteriori error bound]\label{thm:gronwall}
  Let $\cF$ be a $\lambda$-geodesically convex functional on $\Ptwospace[\bR^d]$, $\prt{\mu_t}_{t\in[0,t_f]}$ be the corresponding WGF, and $\prt{\dec_t}_{t\in[0,t_f]}\in AC\prt{[0,t_f],L^2_{\mu_0}}$ be a curve of diffeomorphisms solving \Cref{eq:projected_dyn} such that $\nu_t=\pushfwd{\dec_t}{\mu_0}$.
  Then, for a.e.~$t\in[0,t_f]$ and given the residual
  $$
    r_t\coloneqq v_{\nu_t}\circ\dec_t-\dot{\dec}_t,$$
  the following holds:
  \begin{align*}
    \Wtwometric[\mu_t][\nu_t]
    \leq
    \int_0^t e^{-\lambda\prt{t-s}}\norm{r_s}_{L^2_{\mu_0}}\ds .
  \end{align*}
\end{theorem}

\begin{proof}
  By \Cref{proposition:proj_continuity_eq}, the curve $\prt{\nu_t}_{t\in[0,t_f]}$ satisfies the
  continuity equation for the velocity field
  $\tilde{v}_{\nu_t}
    =
    \proj{\bS_t^v}{L^2_{\nu_t}}
    \prt{v_{\nu_t}}$.
  Hence, applying the two-curve derivative estimate in \Cref{lemma:dt_W2} to
  $\mu_t$ and $\nu_t$, and using the sign convention of the continuity equation
  above, we get
  \begin{align*}
    \frac{\d}{\dt}\frac12 \Wtwometric[\mu_t][\nu_t][2]
     & \leq
    \int_{\bR^d}
    \inner[\big]{
      v_{\mu_t}\prt{x}
      -
      \tilde{v}_{\nu_t}\circ\Toptmap{\mu_t}{\nu_t}\prt{x},
      x-\Toptmap{\mu_t}{\nu_t}\prt{x}
    }
    \d\mu_t\prt{x} \\
     & =
    \int_{\bR^d}
    \inner[\big]{
      v_{\mu_t}\prt{x}
      -
      v_{\nu_t}\circ\Toptmap{\mu_t}{\nu_t}\prt{x},
      x-\Toptmap{\mu_t}{\nu_t}\prt{x}
    }
    \d\mu_t\prt{x} \\
     & \quad+
    \int_{\bR^d}
    \inner[\big]{
      v_{\nu_t}\circ\Toptmap{\mu_t}{\nu_t}\prt{x}
      -
      \tilde{v}_{\nu_t}\circ\Toptmap{\mu_t}{\nu_t}\prt{x},
      x-\Toptmap{\mu_t}{\nu_t}\prt{x}
    }
    \d\mu_t\prt{x}.
  \end{align*}
  The first term is bounded by $-\lambda \Wtwometric[\mu_t][\nu_t][2]$ by the
  $\lambda$-monotonicity of the Wasserstein subdifferential
  recalled in \Cref{eq:geodesic-subdifferential-monotonicity}. The second term is
  bounded by
  $\Wtwometric[\mu_t][\nu_t]
    \norm{v_{\nu_t}-\tilde{v}_{\nu_t}}_{L^2_{\nu_t}}$
  by the Cauchy--Schwarz inequality. Since
  $v_{\nu_t}-\tilde{v}_{\nu_t}
    =\prt{\id-\proj{\bS_t^v}{L^2_{\nu_t}}}
    \prt{v_{\nu_t}}$, the pullback identity in \Cref{eq:projection-pullback} yields
  \begin{align*}
    \norm{v_{\nu_t}-\tilde{v}_{\nu_t}}_{L^2_{\nu_t}}
    =
    \norm{\prt{\id-\proj{\bS_t}{L^2_{\mu_0}}}
    \prt{v_{\nu_t}\circ\dec_t}}_{L^2_{\mu_0}}
    =
    \norm{r_t}_{L^2_{\mu_0}}.
  \end{align*}
  Therefore, combining the monotonicity estimate for the first term with the Cauchy--Schwarz bound on the residual term, we get
  \begin{align*}
    \frac{\d}{\dt}\frac12 \Wtwometric[\mu_t][\nu_t][2]
    \leq
    -\lambda \Wtwometric[\mu_t][\nu_t][2]
    +
    \norm{r_t}_{L^2_{\mu_0}} \Wtwometric[\mu_t][\nu_t].
  \end{align*}
  Since $\dec_0=\Id$, then $\Wtwometric[\mu_0][\nu_0]=0$, and the standard integral form of Gronwall's lemma gives the claim.
\end{proof}

\begin{proposition}[Energy dissipation of the projected curve]
  \label{prop:projected-energy-decay}
  Assume that the hypotheses of \Cref{thm:gronwall} hold and, in addition, that
  $v_{\nu_t}=-\subdifferentialelement\prt{\nu_t}\in L^2\prt{0,t_f;L^2_{\nu_t}}$.
  Then for $\mathcal L^1$-a.e.\ $t\in[0,t_f]$, the curve $\prt{\nu_t}_{t\in[0,t_f]}$ dissipates the
  energy according to
  \begin{align*}
    \frac{\d}{\dt}\cF\prt{\nu_t}
    =
    -\norm{\dot\dec_t}_{L^2_{\nu_0}}^2
    =
    -\norm{v_{\nu_t}}_{L^2_{\nu_t}}^2
    +
    \norm{r_t}_{L^2_{\nu_0}}^2
    \leq 0 .
  \end{align*}
\end{proposition}

\begin{proof}
  Since, by assumption, $v_{\nu_t}=-\subdifferentialelement\prt{\nu_t}\in L^2\prt{[0,t_f],L^2_{\nu_t}}$, and by \Cref{proposition:proj_continuity_eq}, $\nu_t$ satisfies the continuity equation with velocity $\tilde{v}_{\nu_t}$, the chain rule applies (see Theorem 10.3.18 in \cite{AGS2008}) so that, for a.e.\ $t\in[0,t_f]$, we have
  \begin{align*}
    \frac{\d}{\dt}\cF\prt{\nu_t}
     & =
    \int_{\bR^d}
    \inner{\subdifferentialelement\prt{\nu_t},\tilde{v}_{\nu_t}}
    \d\nu_t\prt{x} \\
     & =
    -
    \int_{\bR^d}
    \inner[\big]{
    v_{\nu_t},
    \proj{\bS_t^v}{L^2_{\nu_t}}\prt{v_{\nu_t}}
    }
    \d\nu_t\prt{x} \\
     & =
    -
    \norm{
    \proj{\bS_t^v}{L^2_{\nu_t}}\prt{v_{\nu_t}}
    }_{L^2_{\nu_t}}^2=-\norm{\dot\dec_t}_{L^2_{\nu_0}}^2,
  \end{align*}
  where the last two equalities follow from the orthogonality of the projection and from the pullback identity in \Cref{eq:projection-pullback}, respectively.
  Finally, by the orthogonal decomposition, one has
  \begin{align*}
    \norm{v_{\nu_t}}_{L^2_{\nu_t}}^2
    =
    \norm{\proj{\bS_t^v}{L^2_{\nu_t}}\prt{v_{\nu_t}}}_{L^2_{\nu_t}}^2
    +
    \norm{\prt{\id-\proj{\bS_t^v}{L^2_{\nu_t}}}v_{\nu_t}}_{L^2_{\nu_t}}^2,
  \end{align*}
  and applying again the pullback identity to the second term gives
  $\norm{\prt{\id-\proj{\bS_t^v}{L^2_{\nu_t}}}v_{\nu_t}}_{L^2_{\nu_t}}
    =
    \norm{r_t}_{L^2_{\nu_0}}$,
  which yields the final identity.
\end{proof}

\section{Taylor-based background space}
\label{sec:tb_background}

The choice of $\bW_m$ is flexible and can recover existing approximation strategies.
For instance, if $m=n$ and $\bW_m$ is a finite element
space spanned by standard basis functions, then the method is closely related to dynamical low-rank approximation.

Instead, here we propose a way to build $\bW_m$ by adaptively enriching the decoder with local information on the velocity field. Although the idea
is general, we present it in the setting of the Lagrangian system for energies of the form \Cref{eq:prototypical-energy}.
We therefore consider the Lagrangian velocity operator
\begin{align}\label{eq:lagrangian_velocity_operator}
  \cL:H^{k+2}_{\mu_0} & \to H^k_{\mu_0}, \qquad k\geq0,\nonumber                             \\
  \cT                 & \mapsto -\subdifferentialelement\prt{\pushfwd{\cT}{\mu_0}}\circ \cT,
\end{align}
where the two-derivative loss reflects the second-order character of $\cL$ when the
internal energy $U(\rho)$ is present, i.e.\ when the WGF contains diffusive
terms.
In particular, for $k=0$ we assume that $\cL$ is twice Fréchet
differentiable as a map $H^2_{\mu_0}\to L^2_{\mu_0}$.
We remark that for parabolic regularity for linear diffusion
$U(z)=z\log z$ and porous-medium diffusion $U(z)=z^m/(m-1)$, under suitable regularity assumptions on $V,W$ and $\rho_0$ one has
$(T_t)_{t\in[0,t_f]}\in C([0,t_f];H^2_{\mu_0})$.

\subsection{Construction}
To construct the Taylor-based background space for approximating the transport map with $H^2_{\mu_0}$ regularity, we first consider a generic starting time $t^*\geq 0$ and identify a \textit{starting map} $\cT^*$, represented by the Stiefel decoder as
\begin{align}\label{eq:starting_map}
  \cT^{*}=\sum_{i=1}^n c_i v_i,
  \qquad
  \stiefelframe{v}=(v_1,\dots,v_n)\in\St(n,H^2_{\mu_0};L^2_{\mu_0}),
\end{align}
a \textit{linearization map} $\bar\cT$, and a set of \textit{perturbation directions} $p_1,\dots,p_P$.
Note that the frame in \Cref{eq:starting_map} belongs to $H^2_{\mu_0}$, while orthonormality is
imposed with respect to the $L^2_{\mu_0}$ inner product, consistently with the
metric used for transport-map velocities when applying the dynamical approximation method (see \Cref{sec:stiefel}).
By taking the Lagrangian operator to be Fréchet differentiable, and assuming $\bar\cT,p_1,\dots,p_P\in H^4_{\mu_0}$, then by \Cref{eq:lagrangian_velocity_operator} with $k=2$ we get $\cL\prt{\bar\cT}\in H^2_{\mu_0}$ and $D\cL\prt{\bar\cT}[p_i]\in H^2_{\mu_0}$.
Furthermore, for $\cT\in H^2_{\mu_0}$ sufficiently close to $\bar\cT$, it holds that
\begin{equation}\label{eq:taylor_exp}
  \cL\prt{\cT}
  =
  \cL\prt{\bar \cT}
  +
  D\cL\prt{\bar \cT}[\cT-\bar \cT]
  +R_2(\cT),
  \qquad\text{with}\quad
  \norm{R_2(\cT)}_{L^2_{\mu_0}}
  =
  \mathrm{O}\prt[\big]{\norm{\cT-\bar\cT}_{H^2_{\mu_0}}^2}.
\end{equation}
The Taylor-based background space associated with
$(\cT^*,\bar\cT, \{p_i\}_{i=1}^P)$ is then defined as
\begin{equation}
  \bW_m
  \coloneqq
  \vspan\set{
    v_1,\dots,v_n,\,
    \cL\prt{\bar \cT},\,
    D\cL\prt{\bar \cT}[p_1],\dots,
    D\cL\prt{\bar \cT}[p_P]
  },
\end{equation}
where $m\coloneqq\dim(\bW_m)\le n+P+1$.
Such construction is motivated by the first-order expansion of the Lagrangian velocity around $\bar\cT$ of \Cref{eq:taylor_exp}.
Indeed, if the displacement $\cT-\bar\cT$ is well described by the selected directions $p_i$ and $\cT$ is sufficiently close to $\bar{\cT}$ in $H^2_{\mu_0}$, the linearized velocity $\cL\prt{\bar\cT}+D\cL\prt{\bar\cT}[\cT-\bar\cT]$ is well represented by the span of $\cL\prt{\bar\cT}$ and $D\cL\prt{\bar\cT}[p_i]$.
Moreover, by construction, $\bW_m\subset H^2_{\mu_0}$.

Intuitively, in order for the Taylor-based background space to be effective, the linearization map should remain close to the transport maps in the subsequent time interval, so that $\bW_m$ ensures a controlled error.
This motivates updating the Taylor-based background space at prescribed times.
More precisely, let $\bar{\mathbf t}=\{\bar t_j\}_{j=1}^K$ be a sequence of times realizing a partition of $[0,t_f]$, and let $\bar{\bm{\varepsilon}}=\{\bar{\varepsilon}_j\}_{j=1}^{K-1}$ be a sequence of tolerances.
We say that $(\dec_t)_{t\in[0,t_f]}$ solves \Cref{eq:projected_dyn} with \textit{linearization times} $\bar{\mathbf t}$ and \textit{linearization tolerances} $\bar{\bm{\varepsilon}}$ if the corresponding background space is updated at each $\bar{t}_j\in\bar{\mathbf t}$ by selecting the current map as the starting map,
i.e.\ $\cT^{*}_{\bar{t}_j}=\dec_{\bar{t}_j}$, and an approximation of the current map up to the tolerance $\bar{\varepsilon}_j$ as linearization map $\bar\dec_{\bar t_j}$, while the perturbation directions are kept constant.
In this way, for every $1\le j\le K-1$ and every $t\in[\bar t_j,\bar t_{j+1}]$, the background space is defined as follows:
\begin{align}
  \label{eq:taylor-bk-space}
  \bW_{m,\bar t_j}=\vspan\set{
  v_{\bar{t}_j,1},\dots,v_{\bar{t}_j,n},\,
  \cL\prt{\bar \dec_{\bar{t}_j}},\,
  D\cL\prt{\bar \dec_{\bar{t}_j}}[p_1],\dots,
  D\cL\prt{\bar \dec_{\bar{t}_j}}[p_P]
  }.
\end{align}
To lighten the notation, in the following we will write $\bW_m$ in place of $\bW_{m,\bar t}$ whenever the corresponding linearization time is clear from the context.

\subsection{Error analysis}
In this subsection, we specialize the previous error analysis (see \Cref{thm:gronwall}) to the Taylor-based background space from \cref{eq:taylor-bk-space}.
We restrict to the non-degenerate regime in which $\bc_t\neq 0\ \forall\ t\in[0,t_f]$, as we aim to certify approximations that remain sufficiently close to a non-degenerate exact Lagrangian trajectory.
Accordingly, we exclude the case $\bc_t=0$ throughout.

First, observe that the Taylor-based construction of the background spaces for the linearization times $\bar{\mathbf t}$ and tolerances $\bar{\bm{\varepsilon}}$ implies, for every $1\le j\le K-1$ and $t\in[\bar t_j,\bar t_{j+1}]$, that
\begin{align}\label{eq:relinearization_loss}
  \norm{\dot{\dec}_t-\cL^{\mathrm{app}}_{\bar t_j,\bar{\varepsilon}_j}\prt{\dec_t}}_{L^2_{\mu_0}}=0,
\end{align}
where
\begin{align*}
  \cL^{\mathrm{app}}_{\bar t_j,\bar{\varepsilon}_j}\prt{\dec_t}
  \coloneqq
  \proj{\bW_{m,\bar t_j}}{L^2_{\mu_0}}
  \prt[\Big]{
    \cL\prt{\bar\dec_{\bar t_j}}
    +
    D\cL\prt{\bar\dec_{\bar t_j}}[\dec_t-\bar\dec_{\bar t_j}]
  },
  \qquad\text{with}\quad
  \norm{\dec_{\bar t_j}-\bar\dec_{\bar t_j}}_{H^2_{\mu_0}}<\bar{\varepsilon}_j.
\end{align*}
Then, the residual appearing in \Cref{thm:gronwall} can be decomposed for each $t\in[\bar t_j,\bar t_{j+1}]$ as follows:
\begin{align*}
  \norm{r_t}_{L^2_{\mu_0}}
   & = \norm[\Big]{\cL\prt{\dec_t}-\dot{\dec}_t}_{L^2_{\mu_0}}\leq \norm[\Big]{\cL^{\mathrm{app}}_{\bar{t}_j,\bar{\varepsilon}_j}\prt{\dec_t}-\dot{\dec}_t}_{L^2_{\mu_0}}
  + \norm[\Big]{\cL\prt{\dec_t}-\cL^{\mathrm{app}}_{\bar{t}_j,\bar{\varepsilon}_j}\prt{\dec_t}}_{L^2_{\mu_0}}.
\end{align*}
The first term vanishes by construction by \Cref{eq:relinearization_loss}.
The second one measures the local Taylor truncation error and, by assuming sufficient regularity, gives the following bound:
\begin{align}\label{eq:residual}
  \norm{r_t}_{L^2_{\mu_0}} & \leq \norm[\Big]{(\id-\proj{\bW_{m}}{L^2_{\mu_0}})\prt[\Big]{D\cL\prt{\bar{\dec}_{\bar{t}_j}}[\dec_t-\bar{\dec}_{\bar{t}_j}]}}_{L^2_{\mu_0}}
  + C_{\bar{t}_j,2}\norm{\dec_t-\bar{\dec}_{\bar{t}_j}}^{2}_{H^2_{\mu_0}}  \nonumber                                                                                      \\
                           & \leq C_{\bar t_j,1}\,\norm{e_{\bm{p},\bar t_j,t}}_{H^2_{\mu_0}}+ C_{\bar{t}_j,2}\norm{\dec_t-\bar{\dec}_{\bar{t}_j}}^{2}_{H^2_{\mu_0}},
\end{align}
where $\bP\coloneqq\Span\{p_1,\dots,p_P\}$,
$e_{\bm{p},\bar t_j,t} \coloneqq \prt{\id-\proj{\bP}{H^2_{\mu_0}}}\prt{\dec_t-\bar{\dec}_{\bar{t}_j}}$ is the component of the displacement not captured by the local Taylor basis, and $C_{\bar{t}_j,i}\coloneqq \frac{1}{i!}\norm{D^i\cL\prt{\bar{\dec}_{\bar{t}_j}}}_{\cL^i(H^2_{\mu_0};L^2_{\mu_0})}$, with $\cL^i\prt{H^2_{\mu_0};L^2_{\mu_0}}$ denoting the space of continuous $i$-linear maps from $\prt{H^2_{\mu_0}}^{\otimes i}$ to $L^2_{\mu_0}$.
Moreover, the second term on the right-hand side accounts for the bound of the Taylor residual $R_2(\dec_t)$.
Throughout the rest of this subsection we assume the following hypotheses:
\begin{enumerate}
  \item[(H1)] $\cF$ is of the form in \Cref{eq:prototypical-energy}.
  \item[(H2)] $\cF$ is a $\lambda$-convex functional along generalized geodesics on $\Wtwospace[\bR^d]$, with $\lambda>0$.
  \item[(H3)] The restriction of
    $\cL$ on the affine space $\bW_{m,\bar t_j}+\bar{\dec}_{\bar{t}_j}$ is $C^2$ for each $\bar{t}_j\in\bar{\mathbf t}$.
    Moreover, the constants $C_{\bar t_j,1}$ and $C_{\bar t_j,2}$ are uniformly bounded, with
    $C_{\bar t_j,2}>0$.
\end{enumerate}
Under these assumptions, first, \Cref{lemma:local-residual-control} shows that, on each subinterval, the residual contribution in \Cref{eq:residual}
can be made arbitrarily small by taking both the interval and the linearization tolerance small enough.
Then, \Cref{thm:final-residual} propagates this local control along the whole
partition $[0,t_f]$, showing that the Wasserstein error at the final time can be made arbitrarily small by acting on $\bar{\mathbf t}$ and $\bar{\bm{\varepsilon}}$.
Finally, \Cref{cor:AC} records the absolute-continuity property of the approximate map, and \Cref{cor:subdiff-integrability} proves the integrability of the subdifferential at $\partial^\circ\cF(\nu_t)$ required by \Cref{prop:projected-energy-decay}.


\begin{lemma}[Local residual control]\label{lemma:local-residual-control}
  Fix $\bar t\ge0$. Then, under the assumptions $(H1)$--$(H3)$, for every $\eta>0$, there exist
  $t_+>\bar t$ and $\bar{\varepsilon}_{\bar t}>0$ such that, whenever
  $\bW_m$ is the Taylor-based background space built at $\bar t$ from the
  starting map $\cT^*_{\bar t}=\dec_{\bar t}$ and a linearization map
  $\bar\dec_{\bar t}\in H^4_{\mu_0}$ satisfying
  $\norm{\dec_{\bar t}-\bar\dec_{\bar t}}_{H^2_{\mu_0}}
    \le \bar{\varepsilon}_{\bar t}$, the corresponding trajectory $\dec_t$
  solving \Cref{eq:projected_dyn} on $[\bar t,t_+]$ satisfies
  \begin{align*}
    \int_{\bar t}^{t_+}
    \norm{r_s}_{L^2_{\mu_0}}
    e^{-\lambda\prt{t_+-s}}
    \ds
    \le \eta .
  \end{align*}
\end{lemma}
\begin{proof}
  Fix $t\in[\bar t,t_+]$.
  By taking the exact Lagrangian solution $(T_t)_{t\geq 0}$, we get
  \begin{align}\label{eq:decomposition_T}
    \norm{\dec_t-\bar\dec_{\bar t}}_{L^2_{\mu_0}}\leq \norm{\dec_t-T_t}_{L^2_{\mu_0}}+ \norm{T_t-T_{\bar t}}_{L^2_{\mu_0}}+\norm{T_{\bar t}-\bar\dec_{\bar t}}_{L^2_{\mu_0}}.
  \end{align}
  We start bounding the second term by noting
  \begin{align*}
    \frac{1}{2}\frac{\d}{\dt}\norm{T_t-T_{\bar t}}_{L^2_{\mu_0}}^2=\inner{T_t-T_{\bar t},v_{\mu_t}\circ T_t}_{L^2_{\mu_0}}\leq \norm{T_t-T_{\bar t}}_{L^2_{\mu_0}}\norm{v_{\mu_t}}_{L^2_{\mu_t}},
  \end{align*}
  so that,
  using \Cref{eq:metric-derivative-identity,eq:metric-derivative-decay}, we obtain:
  \begin{align*}
    \norm{T_t-T_{\bar t}}_{L^2_{\mu_0}}\leq \int_{\bar t}^t \norm{v_{\mu_s}}_{L^2_{\mu_s}}\ds\leq\int_{\bar t}^t\abs{\mu'}\prt{s}\ds\leq\abs{\mu'}\prt{\bar t}\int_{\bar t}^te^{-\lambda \prt{s-\bar t}}\ds.
  \end{align*}
  Moreover, since $v_{\nu_t}\circ\dec_t=\tilde{v}_{\nu_t}\circ\dec_t+r_t$, we get
  \begin{align*}
    \frac{1}{2}\frac{\d}{\dt}\norm{\dec_t-T_t}_{L^2_{\mu_0}}^2
     & = \inner[\big]{\dec_t-T_t,\,
    v_{\nu_t}\circ\dec_t-v_{\mu_t}\circ T_t - r_t}_{L^2_{\mu_0}} \\
     & \le -\lambda\norm{\dec_t-T_t}_{L^2_{\mu_0}}^2
    + \norm{r_t}_{L^2_{\mu_0}}\norm{\dec_t-T_t}_{L^2_{\mu_0}},
  \end{align*}
  by Cauchy--Schwarz and
  by the monotonicity estimate induced by $\lambda$-convexity along
  generalized geodesics, recalled in
  \Cref{eq:wasserstein-subdifferential-monotonicity} and applied with base
  measure $\mu_0$ and maps $T_t$ and $\dec_t$, so that
  \begin{align*}
    \inner[\big]{\dec_t-T_t,\,
    v_{\nu_t}\circ\dec_t-v_{\mu_t}\circ T_t}_{L^2_{\mu_0}}
    \le -\lambda\norm{\dec_t-T_t}_{L^2_{\mu_0}}^2.
  \end{align*}
  By Grönwall's inequality,
  \begin{align}\label{eq:bound_on_T}
    \norm{\dec_t-T_t}_{L^2_{\mu_0}}
    \le
    \norm{\dec_{\bar t}-T_{\bar t}}_{L^2_{\mu_0}} e^{-\lambda\prt{t-\bar t}}
    +
    \int_{\bar t}^t \norm{r_s}_{L^2_{\mu_0}} e^{-\lambda\prt{t-s}}\,ds,
  \end{align}
  and plugging everything into \Cref{eq:decomposition_T} yields
  \begin{align*}
    \norm{\dec_t-\bar\dec_{\bar t}}_{L^2_{\mu_0}}
     & \leq \norm{T_{\bar t}-\dec_{\bar t}}_{L^2_{\mu_0}}\prt[\big]{1+e^{-\lambda\prt{t-\bar t}}}
    +\norm{\dec_{\bar t}-\bar\dec_{\bar t}}_{L^2_{\mu_0}}                                         \\
     & \quad
    +\int_{\bar t}^t\norm{r_s}_{L^2_{\mu_0}}e^{-\lambda\prt{t-s}}\ds
    +\abs{\mu'}\prt{\bar t}\int_{\bar t}^te^{-\lambda \prt{s-\bar t}}\ds.
  \end{align*}
  Introducing the quantity $\Phi_{\bar t,t}\coloneqq \int_{\bar t}^t\norm{r_s}_{L^2_{\mu_0}}e^{-\lambda\prt{t-s}}\ds$, by \Cref{eq:residual} we get
  \begin{align*}
    \norm{r_t}_{L^2_{\mu_0}} & \leq \tilde{C}_{\bar{\varepsilon}_{\bar t},\bar t,t}+C_{\bar t,2}\prt[\Big]{A_{\bar{\varepsilon}_{\bar t},\bar t,t}+\Phi_{\bar t,t}}^2,
  \end{align*}
  with
  \begin{align}
    A_{\bar{\varepsilon}_{\bar t},\bar t,t}         & \coloneqq \norm{T_{\bar t}-\dec_{\bar t}}_{L^2_{\mu_0}}\prt[\big]{1+e^{-\lambda\prt{t-\bar t}}}+\norm{\dec_{\bar t}-\bar\dec_{\bar t}}_{L^2_{\mu_0}}+\abs{\mu'}\prt{\bar t}\int_{\bar t}^te^{-\lambda \prt{s-\bar t}}\ds,\label{def:A} \\
    \tilde{C}_{\bar{\varepsilon}_{\bar t},\bar t,t} & \coloneqq C_{\bar t,1}\norm{e_{\bm{p},\bar t,t}}_{H^2_{\mu_0}}+C_{\bar t,2}\prt[\big]{\norm{\nabla\dec_t-\nabla\bar\dec_{\bar t}}^2_{L^2_{\mu_0}}+\norm{\nabla^2\dec_t-\nabla^2\bar\dec_{\bar t}}^2_{L^2_{\mu_0}}},\label{def:C_tilde}
  \end{align}
  so that
  \begin{align*}
    \Phi_{\bar t,t} & \leq \int_{\bar t}^t\prt[\Big]{\tilde{C}_{\bar{\varepsilon}_{\bar t},\bar t,s}+C_{\bar t,2}\prt[\big]{A_{\bar{\varepsilon}_{\bar t},\bar t,s}+\Phi_{\bar t,s}}^2}e^{-\lambda\prt{t-s}}\ds.
  \end{align*}
  Then, by considering $\Psi_{\bar t,t_+}\coloneqq\sup_{t\in[\bar t,t_+]}\Phi_{\bar t,t}$, we obtain
  \begin{align}\label{eq:Psi_ineq}
    \Psi_{\bar t,t_+} & \leq \sup_{t\in[\bar t,t_+]} \prt[\Big]{\tilde{C}_{\bar{\varepsilon}_{\bar t},\bar t,t}+C_{\bar t,2}\prt[\big]{A_{\bar{\varepsilon}_{\bar t},\bar t,t}+\Phi_{\bar t,t}}^2}\int_{\bar t}^{t_+}e^{-\lambda\prt{t_+-s}}\ds\nonumber                                                             \\
                      & \leq \prt[\Big]{C_{\bar t,2}\Psi_{\bar t,t_+}^2+2C_{\bar t,2}A^{*}_{\bar{\varepsilon}_{\bar t},\bar t,t_+}\Psi_{\bar t,t_+}+\tilde{C}^{*}_{\bar{\varepsilon}_{\bar t},\bar t,t_+}+C_{\bar t,2}A^{*^2}_{\bar{\varepsilon}_{\bar t},\bar t,t_+}}\int_{\bar t}^{t_+}e^{-\lambda\prt{t_+-s}}\ds,
  \end{align}
  with $A^{*}_{\bar{\varepsilon}_{\bar t},\bar t,t_+}\coloneqq\sup_{t\in[\bar t,t_+]}A_{\bar{\varepsilon}_{\bar t},\bar t,t}$, and $\tilde{C}^{*}_{\bar{\varepsilon}_{\bar t},\bar t,t_+}\coloneqq\sup_{t\in[\bar t,t_+]}\tilde{C}_{\bar{\varepsilon}_{\bar t},\bar t,t}$.
  Applying \Cref{lemma:finite-dimensional-AC}, the projected trajectory satisfies
  $\dec\in AC([\bar t,t_+];H^2_{\mu_0})$.
  In particular, $t\mapsto\dec_t$ is continuous in $H^2_{\mu_0}$ at $\bar t$.
  Hence, since the map
  $t\mapsto\tilde{C}_{\bar{\varepsilon}_{\bar t},\bar t,t}$
  is continuous, for every $\epsilon>0$ one can choose $t_+-\bar t$ small enough so that
  \begin{align*}
    \tilde{C}^{*}_{\bar{\varepsilon}_{\bar t},\bar t,t_+}
     & \leq
    C_{\bar t,1}
    \norm{\prt{\id-\proj{\bP}{H^2_{\mu_0}}}\prt{\dec_{\bar t}-\bar\dec_{\bar t}}}_{H^2_{\mu_0}} +
    C_{\bar t,2}
    \prt[\Big]{
    \norm{\nabla\dec_{\bar t}-\nabla\bar\dec_{\bar t}}^2_{L^2_{\mu_0}}
    +
    \norm{\nabla^2\dec_{\bar t}-\nabla^2\bar\dec_{\bar t}}^2_{L^2_{\mu_0}}
    }
    +\epsilon \\
     & \leq
    C_{\bar t,1}\bar{\varepsilon}_{\bar t}
    +
    C_{\bar t,2}\bar{\varepsilon}_{\bar t}^2
    +\epsilon .
  \end{align*}
  By grouping the terms in \Cref{eq:Psi_ineq}, we identify the coefficients $K_{\bar t,t_+},L_{\bar{\varepsilon}_{\bar t},\bar t,t_+},R_{\bar t,t_+}$ realizing
  \begin{align*}
    \Psi_{\bar t,t_+} \leq K_{\bar t,t_+}\Psi_{\bar t,t_+}^2+R_{\bar t,t_+}\Psi_{\bar t,t_+}+L_{\bar{\varepsilon}_{\bar t},\bar t,t_+}.
  \end{align*}
  From the previous estimate, we have
  $\tilde C^*_{\bar{\varepsilon}_{\bar t},\bar t,t_+}\to0$
  as $\bar{\varepsilon}_{\bar t}\to0$ and $t_+-\bar t\to0$.
  Moreover, $A^*_{\bar{\varepsilon}_{\bar t},\bar t,t_+}$ remains bounded for
  $\bar{\varepsilon}_{\bar t}$ and $t_+-\bar t$ small enough.
  Consequently $K_{\bar t,t_+}$, $R_{\bar t,t_+}$ and
  $L_{\bar{\varepsilon}_{\bar t},\bar t,t_+}$ can be made small enough so that
  $R_{\bar t,t_+}<1$ and
  $4K_{\bar t,t_+}L_{\bar{\varepsilon}_{\bar t},\bar t,t_+}<
    (1-R_{\bar t,t_+})^2$.

  If $4 K_{\bar t,t_+} L_{\bar{\varepsilon}_{\bar t},\bar t,t_+} < (1-R_{\bar t,t_+})^2$ and $R_{\bar t,t_+}<1$, then a bootstrapping argument yields
  \begin{align}\label{eq:Psi_bound}
    0\leq\Psi_{\bar t,t_+}<\frac{(1-R_{\bar t,t_+})-\sqrt{(1-R_{\bar t,t_+})^2-4K_{\bar t,t_+}L_{\bar{\varepsilon}_{\bar t},\bar t,t_+}}}{2K_{\bar t,t_+}},
  \end{align}
  since $\Psi_{\bar t,\bar t}=0$.
  By choosing $t_+-\bar t$ and $\bar{\varepsilon}_{\bar t}$ small enough, the
  right-hand side can be made smaller than $\eta$, so that
  \begin{align*}
    \int_{\bar t}^{t_+}
    \norm{r_s}_{L^2_{\mu_0}}
    e^{-\lambda\prt{t_+-s}}
    \ds
    \le
    \Psi_{\bar t,t_+}
    \le
    \eta.
  \end{align*}
\end{proof}

\begin{theorem}[Final-time accuracy of the Taylor scheme]\label{thm:final-residual}
  Under the assumptions $(H1)$--$(H3)$, for every final time $t_f>0$ and
  every accuracy $\eta>0$, there exist a finite partition
  $\bar{\mathbf t}=\{0=\bar t_1<\cdots<\bar t_K=t_f\}$ of $[0,t_f]$
  and tolerances
  $\bar{\bm{\varepsilon}}=\{\bar{\varepsilon}_j\}_{j=1}^{K-1}$
  such that, whenever the
  background space $\bW_{m,\bar t_j}$ is built on each interval $[\bar t_j,\bar t_{j+1}]$, the corresponding projected trajectory
  $\dec_t$ solving \Cref{eq:projected_dyn} with initial datum $\dec_0=\Id$
  satisfies, for $\nu_t\coloneqq\pushfwd{\dec_t}{\mu_0}$,
  \begin{align*}
    \Wtwometric[\mu_{t_f}][\nu_{t_f}]\le\eta .
  \end{align*}
\end{theorem}
\begin{proof}
  First, we claim that the time partition used below can be chosen finite, and we apply \Cref{lemma:local-residual-control} on each relinearization interval
  and propagate the resulting local residual bounds.
  For $j\ge1$, define
  \begin{align*}
    S_j\coloneqq
    \sum_{k=1}^{j-1}
    \Psi_{\bar{t}_k,\bar{t}_{k+1}}
    e^{-\lambda\prt{\bar{t}_j-\bar{t}_{k+1}}}.
  \end{align*}
  Recalling \Cref{def:A}, we get
  \begin{align*}
    A^{*}_{\bar{\varepsilon}_j,\bar{t}_j,\bar{t}_{j+1}}
    \coloneqq \sup_{t\in[\bar{t}_j,\bar{t}_{j+1}]}A_{\bar{\varepsilon}_j,\bar{t}_j,t}
     & \leq 2\norm{T_{\bar{t}_j}-\dec_{\bar{t}_j}}_{L^2_{\mu_0}}
    +\norm{\dec_{\bar{t}_j}-\bar\dec_{\bar{t}_j}}_{L^2_{\mu_0}} +\abs{\mu'}\prt{\bar{t}_j}
    \int_{\bar{t}_j}^{\bar{t}_{j+1}}e^{-\lambda\prt{s-\bar{t}_j}}\, ds                               \\
     & \leq 2\int_0^{\bar{t}_j}\norm{r_s}_{L^2_{\mu_0}}e^{-\lambda\prt{\bar{t}_j-s}}ds
    +\norm{\dec_{\bar{t}_j}-\bar\dec_{\bar{t}_j}}_{L^2_{\mu_0}} +\abs{\mu'}\prt{\bar{t}_j}
    \int_{\bar{t}_j}^{\bar{t}_{j+1}}e^{-\lambda\prt{s-\bar{t}_j}}\, ds                               \\
     & \leq 2\sum_{k=1}^{j-1}\Psi_{\bar{t}_k,\bar{t}_{k+1}}e^{-\lambda\prt{\bar{t}_j-\bar{t}_{k+1}}}
    +\norm{\dec_{\bar{t}_j}-\bar\dec_{\bar{t}_j}}_{L^2_{\mu_0}} +\abs{\mu'}\prt{\bar{t}_j}
    \int_{\bar{t}_j}^{\bar{t}_{j+1}}e^{-\lambda\prt{s-\bar{t}_j}}\, ds,
  \end{align*}
  Denoting by $C_{\mathrm{emb}}>0$ the embedding constant of $H^2_{\mu_0}\hookrightarrow L^2_{\mu_0}$, it follows that
  $A^{*}_{\bar{\varepsilon}_j,\bar{t}_j,\bar{t}_{j+1}}\leq 2S_j+C_{\mathrm{emb}}\bar{\varepsilon}_j+\abs{\mu'}\prt{\bar{t}_j}\prt{\bar{t}_{j+1}-\bar{t}_j}+o\prt{\bar{t}_{j+1}-\bar{t}_j}$.
  Moreover,
  \begin{align*}
    K_{\bar{t}_j,\bar{t}_{j+1}}                     & =C_{\bar{t}_j,2}\prt{\bar{t}_{j+1}-\bar{t}_j}+ o\prt{\bar{t}_{j+1}-\bar{t}_j},                                                                                                  \\
    R_{\bar{t}_j,\bar{t}_{j+1}}                     & =2C_{\bar{t}_j,2}\prt{2S_j+\bar{\varepsilon}_j}\prt{\bar{t}_{j+1}-\bar{t}_j}+o\prt{\bar{t}_{j+1}-\bar{t}_j},                                                                    \\
    L_{\bar{\varepsilon}_j,\bar{t}_j,\bar{t}_{j+1}} & =\prt{\tilde{C}^{*}_{\bar{\varepsilon}_j,\bar{t}_j,\bar{t}_{j+1}}+C_{\bar{t}_j,2}\prt{2S_j+\bar{\varepsilon}_j}^2}\prt{\bar{t}_{j+1}-\bar{t}_j}+o\prt{\bar{t}_{j+1}-\bar{t}_j}.
  \end{align*}
  Using the quantitative estimate obtained in the proof of
  \Cref{lemma:local-residual-control} on $[\bar t_j,\bar t_{j+1}]$,
  namely the explicit bound \Cref{eq:Psi_bound}, we obtain
  \begin{align*}
    \Psi_{\bar{t}_j,\bar{t}_{j+1}}\leq \prt{\tilde{C}^{*}_{\bar{\varepsilon}_j,\bar{t}_j,\bar{t}_{j+1}}+C_{\bar{t}_j,2}\prt{2S_j+\bar{\varepsilon}_j}^2}\prt{\bar{t}_{j+1}-\bar{t}_j}+o\prt{\bar{t}_{j+1}-\bar{t}_j}.
  \end{align*}
  By definition, $S_j$ satisfies the recursive relation
  $S_{j+1}=S_je^{-\lambda\prt{\bar{t}_{j+1}-\bar{t}_j}}+\Psi_{\bar{t}_{j},\bar{t}_{j+1}}$, so that
  \begin{align}\label{eq:recursive}
    S_{j+1}\leq S_j +\prt{\tilde{C}^{*}_{\bar{\varepsilon}_j,\bar{t}_j,\bar{t}_{j+1}}+C_{\bar{t}_j,2}\prt{2S_j+\bar{\varepsilon}_j}^2-\lambda S_j}\prt{\bar{t}_{j+1}-\bar{t}_j}+o\prt{\bar{t}_{j+1}-\bar{t}_j}.
  \end{align}
  For each $j$, set
  \begin{align*}
    f_j(s)
    \coloneqq
    \tilde{C}^{*}_{\bar{\varepsilon}_j,\bar{t}_j,\bar{t}_{j+1}}
    +
    C_{\bar{t}_j,2}\prt{2s+\bar{\varepsilon}_j}^2
    -
    \lambda s .
  \end{align*}
  By choosing the partition sufficiently fine and tolerances sufficiently low, we impose,
  for every $j$,
  \begin{align*}
    \prt{4C_{\bar t_j,2}\bar{\varepsilon}_j-\lambda}^2 & >16C_{\bar t_j,2}\prt[\big]{\tilde{C}^{*}_{\bar{\varepsilon}_j,\bar t_j,\bar t_{j+1}}+C_{\bar t_j,2}\bar{\varepsilon}_j^2}, \\
    \lambda                                            & >4C_{\bar t_j,2}\bar{\varepsilon}_j,
  \end{align*}
  so that $f_j$ has two positive roots:
  \begin{align*}
    \bar{C}_{\bar{t}_j}^{(\pm)}
    \coloneqq
    \frac{\lambda-4C_{\bar{t}_j,2}\bar{\varepsilon}_j}{8C_{\bar{t}_j,2}}
    \pm
    \sqrt{
    \prt[\Big]{\frac{\lambda-4C_{\bar{t}_j,2}\bar{\varepsilon}_j}{8C_{\bar{t}_j,2}}}^2
    -
    \frac{
    \tilde{C}^{*}_{\bar{\varepsilon}_j,\bar{t}_j,\bar{t}_{j+1}}
    +
    C_{\bar{t}_j,2}\bar{\varepsilon}_j^2
    }{
    4C_{\bar{t}_j,2}
    }
    } .
  \end{align*}
  By (H3), the constants $C_{\bar t_j,2}$ remain bounded.
  Hence, since $\bar{C}_{\bar t_j}^{(-)}\to 0$ as
  $\bar{\varepsilon}_j,\prt{\bar t_{j+1}-\bar t_j}\to 0$, while the roots
  $\bar{C}_{\bar t_j}^{(+)}$ remain uniformly bounded away from zero, there
  exists $\eta^*>0$ such that, for every $0<\eta_0<\eta^*$, one can choose
  $\bar{\mathbf t}$ and $\bar{\bm{\varepsilon}}$ so that
  \begin{align*}
    \max_j\bar{C}_{\bar t_j}^{(-)}
    <
    \eta_0
    <
    \min_j\bar{C}_{\bar t_j}^{(+)}.
  \end{align*}
  Given the prescribed accuracy $\eta>0$, choose
  $\eta_0\coloneqq\min\{\eta,\eta^*/2\}$.
  For this choice, $f_j(\eta_0)<0$ for every $j$.
  Thus, for each $j$ there exist $\gamma_j>0$ and $\varepsilon_j>0$ such that
  $f_j(s)\le -\gamma_j$ for every
  $s\in[\eta_0-\varepsilon_j,\eta_0]$.

  We prove by induction that $S_j<\eta_0$ for every $j$.
  Since $S_1=0$ and the identity initial map is represented exactly, we may choose $\bar\dec_{\bar t_1}=\dec_0=\Id$ and $\bar{\varepsilon}_1=0$. Hence
  $S_2=\Psi_{\bar t_1,\bar t_2}\to0$ as
  $\bar t_2-\bar t_1\to0$.
  Thus, by choosing the first time step small enough, $S_2<\eta_0$.
  Assume now that $S_j<\eta_0$.
  If $S_j\in[\eta_0-\varepsilon_j,\eta_0]$, then, since the remainder in
  \Cref{eq:recursive} is uniform for bounded $S_j$, by reducing
  $\bar t_{j+1}-\bar t_j$ if necessary we get
  \begin{align*}
    S_{j+1}
    \le
    S_j
    -
    \frac{\gamma_j}{2}
    (\bar t_{j+1}-\bar t_j)
    <
    \eta_0 .
  \end{align*}
  If instead $0\le S_j<\eta_0-\varepsilon_j$, then
  \begin{align*}
    S_{j+1}
    =
    e^{-\lambda\prt{\bar t_{j+1}-\bar t_j}}S_j
    +
    \Psi_{\bar t_j,\bar t_{j+1}}
    \le
    S_j+\Psi_{\bar t_j,\bar t_{j+1}} .
  \end{align*}
  Since $\Psi_{\bar t_j,\bar t_{j+1}}\to0$ as
  $\bar t_{j+1}-\bar t_j\to0$, refining the partition further gives
  $\Psi_{\bar t_j,\bar t_{j+1}}<\varepsilon_j$ and hence
  $S_{j+1}<\eta_0$.
  Finally, recalling \Cref{eq:bound_on_T}:
  \begin{align*}
    \Wtwometric[\mu_{t_f}][\nu_{t_f}]\leq\norm{T_{t_f}-\dec_{t_f}}_{L^2_{\mu_0}}
     & \leq
    \int_0^{t_f}
    \norm{r_s}_{L^2_{\mu_0}}e^{-\lambda\prt{t_f-s}}\,\ds \\
     & \leq
    \sum_{k=1}^{K-1}
    \Psi_{\bar t_k,\bar t_{k+1}}
    e^{-\lambda\prt{t_f-\bar t_{k+1}}}
    = S_K
    <\eta_0
    \leq\eta .
  \end{align*}

  It remains now to prove that the partition can be chosen finite.
  Under the bootstrap condition $S_j\leq\eta_0$, the local estimates in
  \Cref{lemma:local-residual-control} depend on the relinearization time only
  through uniformly bounded quantities: the constants
  $C_{\bar t_j,1},C_{\bar t_j,2}$ by (H3), the metric derivative
  $\abs{\mu'}\prt{\bar t_j}$ by \Cref{eq:metric-derivative-decay}, and the accumulated
  residual $S_j$. Hence there exist $\delta_*>0$ and $\varepsilon_*>0$ such
  that the local construction applies on every interval of length at most
  $\delta_*$ with tolerance at most $\varepsilon_*$. Taking
  $\bar t_{j+1}=\min\{\bar t_j+\delta_*,t_f\}$ and
  $\bar\varepsilon_j\leq\varepsilon_*$ gives a finite partition with
  $K\leq \lceil t_f/\delta_*\rceil+1$.
\end{proof}

\begin{corollary}[Absolute continuity]\label{cor:AC}
  Under the assumptions $(H1)$--$(H3)$, let $\bar{\mathbf t}$ and
  $\bar{\bm{\varepsilon}}$ be the finite partition and
  the tolerances constructed in \Cref{thm:final-residual}.
  Then the corresponding relinearized trajectory satisfies
  $(\dec_t)_{t\in[0,t_f]}\in AC([0,t_f];H^2_{\mu_0})$,
  and, in particular, $(\dec_t)_{t\in[0,t_f]}\in AC([0,t_f];L^2_{\mu_0})$.
\end{corollary}

\begin{proof}
  Fix $j\in\{1,\dots,K-1\}$.
  On the interval $[\bar t_j,\bar t_{j+1}]$, the relinearized trajectory evolves
  in the finite-dimensional space $\bW_{m,\bar t_j}\subset H^2_{\mu_0}$ and
  solves the projected dynamics driven by
  $\proj{\bW_{m,\bar t_j}}{L^2_{\mu_0}}\cL\prt{\dec_t}$.
  Applying \Cref{lemma:finite-dimensional-AC} gives
  \begin{align*}
    (\dec_t)_{t\in[\bar t_j,\bar t_{j+1}]}\in AC([\bar t_j,\bar t_{j+1}];H^2_{\mu_0}).
  \end{align*}
  Since the partition is finite and the relinearized trajectory is continuous
  at the relinearization times, the local absolute-continuity properties glue
  together, yielding $(\dec_t)_{t\in[0,t_f]}\in AC([0,t_f];H^2_{\mu_0})$.
  The continuous embedding $H^2_{\mu_0}\hookrightarrow L^2_{\mu_0}$ then gives
  $(\dec_t)_{t\in[0,t_f]}\in AC([0,t_f];L^2_{\mu_0})$.
\end{proof}

\begin{corollary}[Integrability of the subdifferential along the projected curve]\label{cor:subdiff-integrability}
  Under the assumptions $(H1)$--$(H3)$, let
  $\bar{\mathbf t}$ and $\bar{\bm{\varepsilon}}$ be the finite partition and
  the tolerances constructed in \Cref{thm:final-residual}, and let
  $\nu_t\coloneqq\pushfwd{\dec_t}{\mu_0}$.
  Then
  \begin{align*}
    \int_0^{t_f}
    \norm{\subdifferentialelement\prt{\nu_t}}^2_{L^2_{\nu_t}}
    \dt
    <+\infty .
  \end{align*}
\end{corollary}

\begin{proof}
  By \Cref{cor:AC}, the relinearized trajectory satisfies
  $(\dec_t)_{t\in[0,t_f]}\in AC([0,t_f];H^2_{\mu_0})$, and therefore
  $(\dec_t)_{t\in[0,t_f]}\in C([0,t_f];H^2_{\mu_0})$.
  Fix $j\in\{1,\dots,K-1\}$. By construction,
  $\dec_t\in\bW_{m,\bar t_j}$ for every
  $t\in[\bar t_j,\bar t_{j+1}]$. Hence, from (H3), the map
  $t\mapsto\cL\prt{\dec_t}$ is continuous from
  $[\bar t_j,\bar t_{j+1}]$ to $L^2_{\mu_0}$, and
  \begin{align*}
    \sup_{t\in[\bar t_j,\bar t_{j+1}]}
    \norm{\cL\prt{\dec_t}}_{L^2_{\mu_0}}
    <+\infty .
  \end{align*}
  Since the partition is finite, it follows that
  $\cL\prt{\dec_t}\in L^2(0,t_f;L^2_{\mu_0})$.
  Finally, by the definition of the Lagrangian velocity operator,
  $\norm{\subdifferentialelement\prt{\nu_t}}_{L^2_{\nu_t}}
    =
    \norm{\cL\prt{\dec_t}}_{L^2_{\mu_0}}$
  for a.e.\ $t\in[0,t_f]$, and this concludes the proof.
\end{proof}

\section{Numerical results}
\label{sec:numerical-results}
In this section, we assess the method on linear and nonlinear Fokker--Planck equations,
porous-medium diffusion, and nonlocal interaction dynamics, comparing the
computed solutions with finite element references computed using mesh size $h=0.04$ and time step $\Delta t=2^{-10}$, or analytic expressions, when available.
We remark that, in the following experiments, the background space is built
only once, at the initial time, and then kept fixed throughout the simulation.
Indeed, this simplified setting is already enough for the considered benchmarks to obtain good accuracy with a moderate number of basis functions, and it also allows us to isolate the effect of the taylor-based construction.
When dealing with more demanding test cases, i.e. higher dimensional regimes, the re-linearization strategy previously introduced in the context of \Cref{sec:tb_background} could instead be used to update
the background space along the evolution, potentially keeping the number of
basis functions under control.

\subsection{Linear Fokker--Planck equation}

We start with the linear Fokker--Planck equation associated with the quadratic
potential, namely the Wasserstein gradient flow of
\begin{align*}
  \cF(\mu)
  =
  \int_{\bR^d}
  \left(\rho(x)\log\rho(x)+\frac{1}{2}|x|^2\rho(x)\right)\dx,
  \qquad\text{where }
  \mu=\rho\cL^d .
\end{align*}
This constitutes a natural first benchmark: Gaussian densities remain Gaussian, and the
exact Lagrangian map acting on a Gaussian initial condition is affine with time-dependent slope and offset.
Moreover, $\cF$ is $\lambda$-geodesically convex with convexity constant $\lambda=1$, so that the energy decay is controlled by an exponential rate (see \Cref{eq:energy_decay} in \Cref{app:wasserstein-convexity}).
Since $\lambda>0$,
we display the energy theoretical exponential decay in \Cref{fig:linear_FP_1d} and \Cref{fig:linear_FP_2d} with red triangles, respectively for $1$D and $2$D test cases.

\begin{figure}[hbt!]
  \centering
  \begin{subfigure}[t]{0.62\textwidth}
    \centering
    \includegraphics[width=\textwidth]{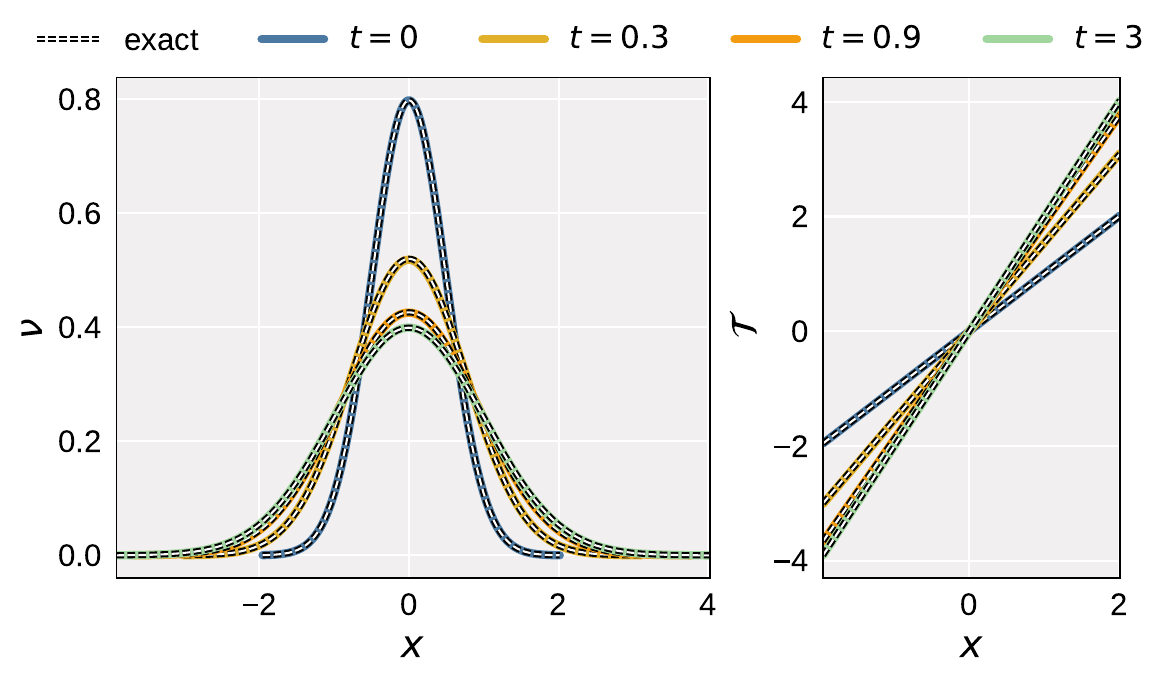}
    \caption{}
    \label{subfig:linear_FP_1d_a}
  \end{subfigure}
  \hfill
  \begin{subfigure}[t]{0.32\textwidth}
    \centering
    \includegraphics[width=\textwidth]{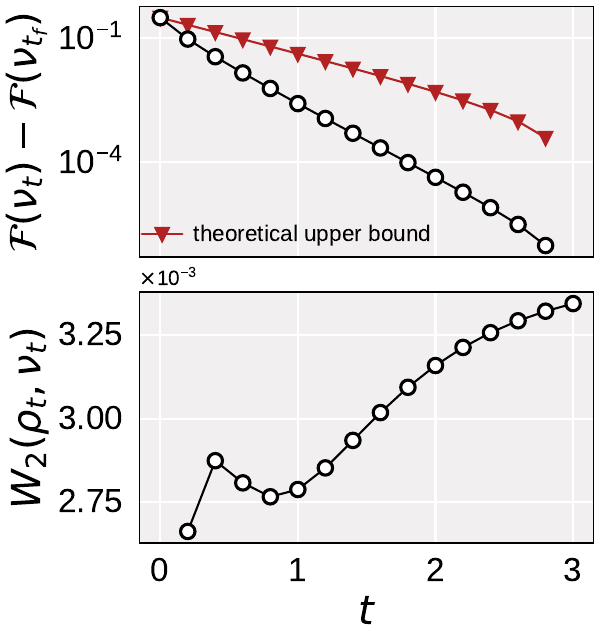}
    \caption{}
    \label{subfig:linear_FP_1d_b}
  \end{subfigure}
  \caption{\footnotesize One-dimensional linear Fokker--Planck equation with
  $n=2$, initialized with a centered Gaussian $\mu_0\sim\cN(0.0,0.5)$. The density $\nu_t$ and the transport map $\dec_t$ are compared with the analytic solutions in \ref{subfig:linear_FP_1d_a}, and the
  Wasserstein error is reported together with the energy decay and its
  theoretical upper bound in \ref{subfig:linear_FP_1d_b}. The results are obtained with a time-step size $\Delta t=2^{-6}$, $t_f=3$ and $N_{\text{samples}}=2^{15}$.}
  \label{fig:linear_FP_1d}
\end{figure}

In this context, \Cref{fig:linear_FP_1d} shows that a two-dimensional adaptive space is enough
to recover the exact one-dimensional dynamics. This is consistent with the
affine form of the transport map. The Wasserstein error remains small over the
whole interval, and the energy decay follows the theoretical bound, confirming
that the dynamical approximation preserves the dissipative behavior of the
gradient flow.

\begin{figure}[hbt!]
  \centering
  \begin{subfigure}[t]{0.62\textwidth}
    \centering
    \includegraphics[width=\textwidth]{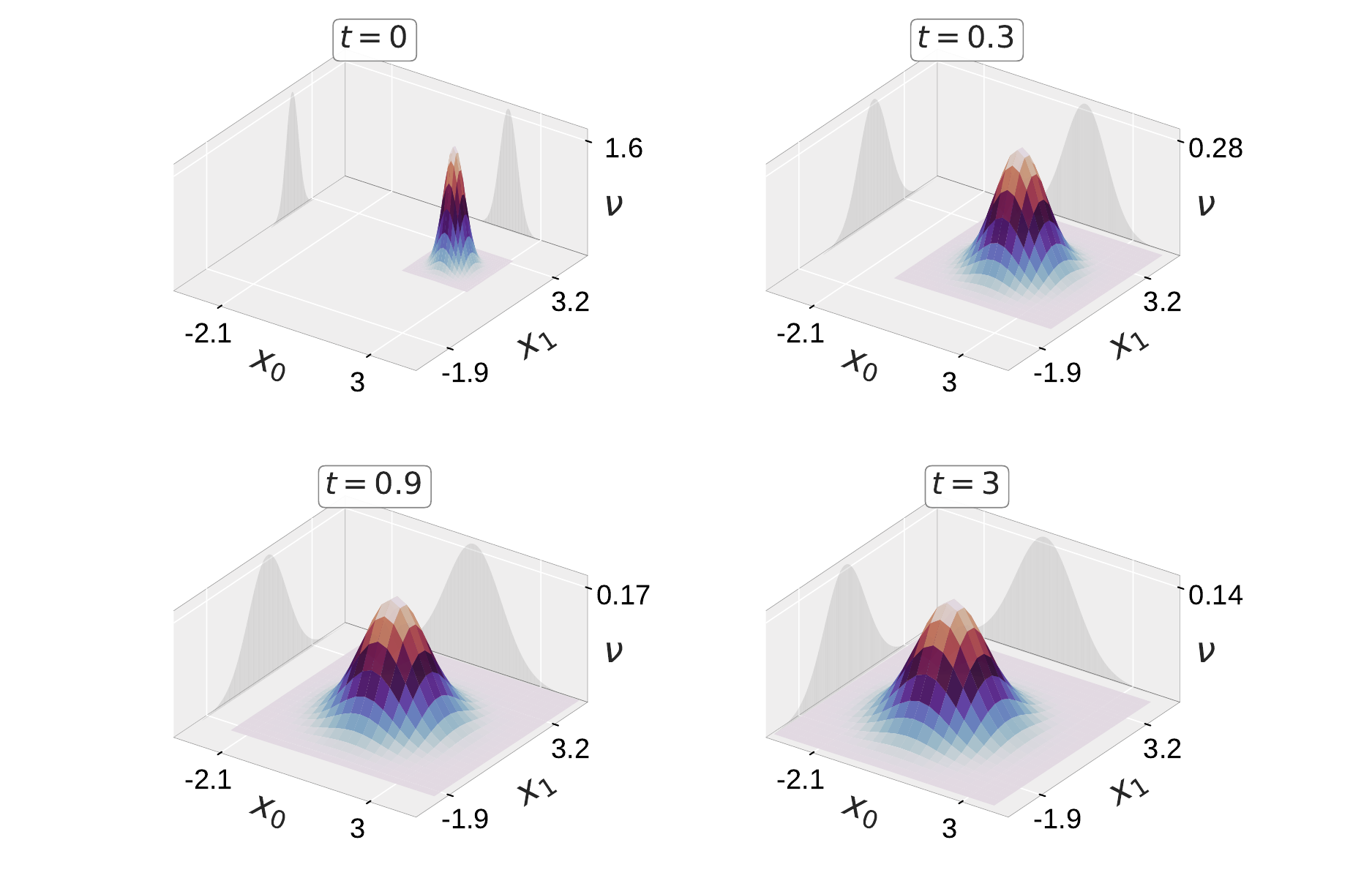}
    \caption{}
    \label{subfig:linear_FP_2d_a}
  \end{subfigure}
  \hfill
  \begin{subfigure}[t]{0.32\textwidth}
    \centering
    \includegraphics[width=\textwidth]{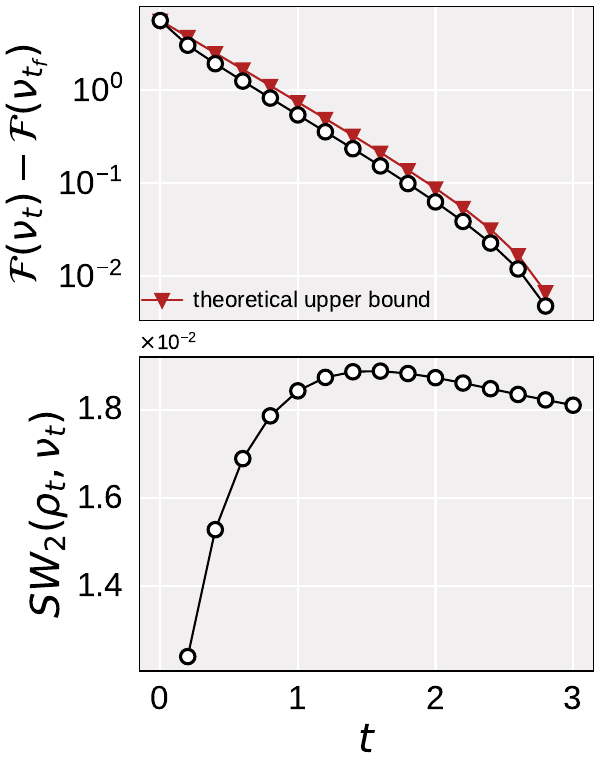}
    \caption{}
    \label{subfig:linear_FP_2d_b}
  \end{subfigure}
  \caption{\footnotesize Two-dimensional linear Fokker--Planck equation with
  $n=6$, initialized with a non-centered Gaussian $\mu_0\sim\cN\big((2.0,2.0),\operatorname{diag}((0.3,0.3))\big)$. The density evolution is displayed in \ref{subfig:linear_FP_2d_a}, and the
  Wasserstein error is reported together with the energy decay and its
  theoretical upper bound in \ref{subfig:linear_FP_2d_b}. The results are obtained with a time-step size $\Delta t=2^{-6}$ and $N_{\text{samples}}=2^{15}$.
      The reference solution is computed with a finite element method.}
  \label{fig:linear_FP_2d}
\end{figure}

The same behavior is observed in dimension two in \Cref{fig:linear_FP_2d}. Starting from a non-centered
Gaussian, the density first moves toward the origin and then relaxes to the
stationary distribution.
Notably, our method does not rely on any underlying mesh structure for the local approximation of the transport map. As a result, extending it to higher dimensions is straightforward, since the main additional requirement is the ability to compute integrals in higher-dimensional spaces, which in our case is handled through Monte Carlo sampling.
In this example, the number of basis functions is still relatively small, but it increases with respect to the one-dimensional case, since in two dimensions there are up to six basis functions of degree at most one, rather than just two.
The error and the energy decay remain consistent with
the exponential bound for $\lambda=1$, showing that the construction extends directly to
the multidimensional linear case.

\begin{figure}[hbt!]
  \centering
  \begin{subfigure}[t]{0.55\textwidth}
    \centering
    \includegraphics[width=\textwidth]{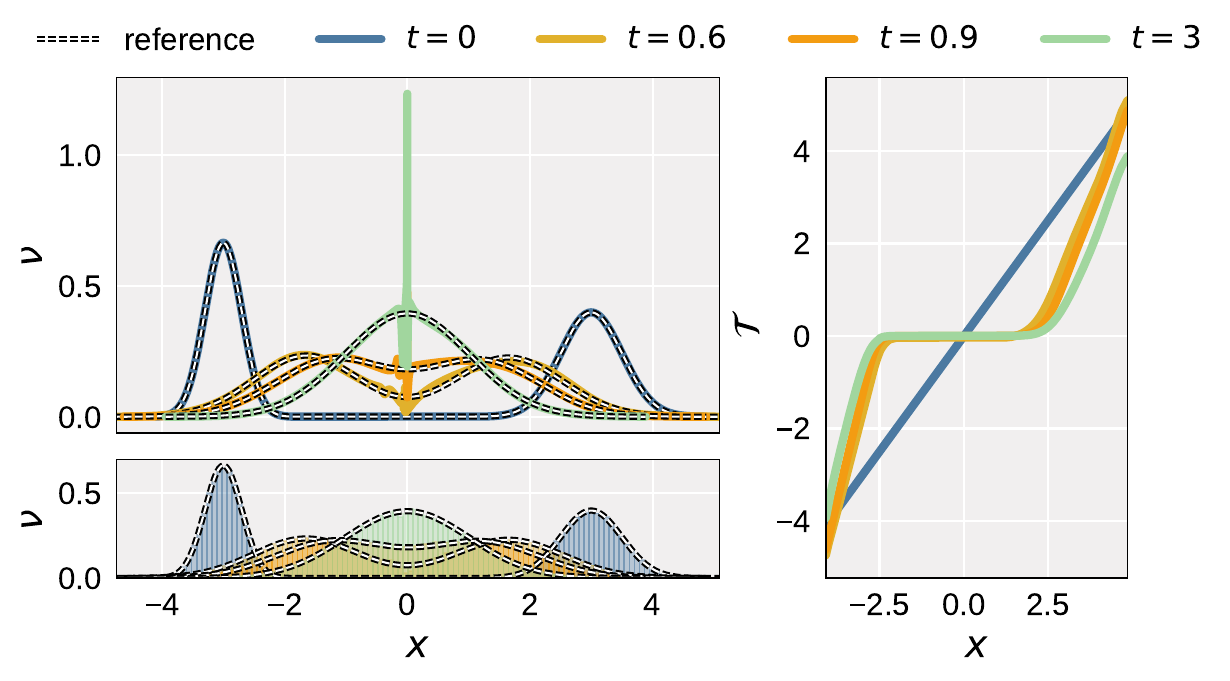}
    \caption{}
    \label{subfig:linear_FP_1d_mg_a}
  \end{subfigure}
  \hfill
  \begin{subfigure}[t]{0.40\textwidth}
    \centering
    \includegraphics[width=\textwidth]{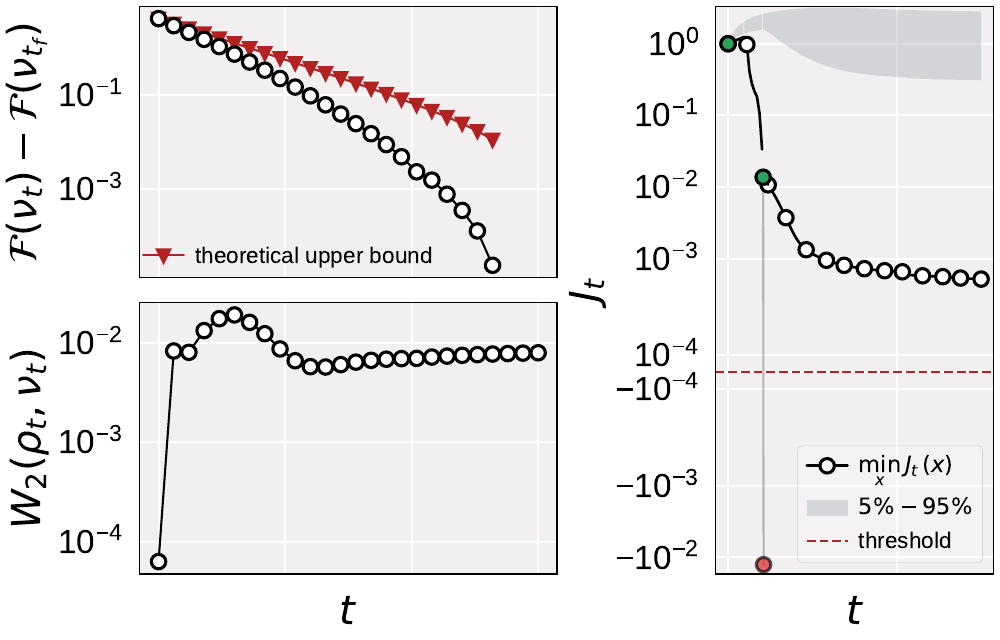}
    \caption{}
    \label{subfig:linear_FP_1d_mg_b}
  \end{subfigure}
  \caption{\footnotesize One-dimensional linear Fokker--Planck equation with $n=8$,
  initialized with $\mu_0\sim \frac12\mathcal N(-3,0.3)+\frac12\mathcal N(3,0.5)$. The density evolution is shown in \ref{subfig:linear_FP_1d_mg_a}: the upper panel is computed using the exact pushforward formula, while the lower one from the histogram of the transported samples.
  The Wasserstein error is reported together with the energy decay and its
  theoretical upper bound in \ref{subfig:linear_FP_1d_mg_b}, as well as the temporal evolution of the spatial minimum and the $5\%$--$95\%$ percentile range of the Jacobian determinant. The results are obtained with a time-step size $\Delta t=2^{-10}$ and $N_{\text{samples}}=2^{16}$.}
\end{figure}

From this experiment onward, whenever no analytic solution is available,
the Wasserstein error is measured against a finite element reference solution.
Thus, the reported error also reflects the accuracy of the reference
discretization, including the effects of mesh size, time stepping, and domain
truncation.

We finally consider a more demanding one-dimensional example, with an initial
mixture of Gaussians. Here the density components merge during the evolution,
and the transport map becomes close to singular, as it can be seen from \Cref{subfig:linear_FP_1d_mg_a}.
For this reason, we include a Jacobian diagnostic by displaying the minimum, over a
finite set of control points, of the determinant of the Jacobian of the
transport map. When this value crosses the invertibility threshold (red
markers), the computation is rolled back to an earlier safe time (green
markers) and restarted from there. The global map is then obtained by composing
the maps generated on the successive time intervals, so that the push-forward
formula remains well-defined on the whole time interval. In this way, we can exploit the aforementioned quantity as a numerical diagnostic to detect a-posteriori lack of invertibility of the decoder.

Small oscillations can be observed near the merging region in
\Cref{subfig:linear_FP_1d_mg_a}. These are consistent with the small values of
the Jacobian appearing in the push-forward formula: as its determinant
approaches zero, even small approximation errors may be amplified at the density
level. However, this effect is highly localized and does not translate into a
visible deterioration of the Wasserstein error. Indeed, the error remains well
controlled, and the empirical histograms obtained from the transported samples
are in very good agreement with the finite element reference solution in the bottom left panel.

\FloatBarrier

\subsection{Double-well Fokker--Planck equation}
\begin{figure}[hbt!]
  \centering
  \begin{subfigure}[t]{0.62\textwidth}
    \centering
    \includegraphics[width=\textwidth]{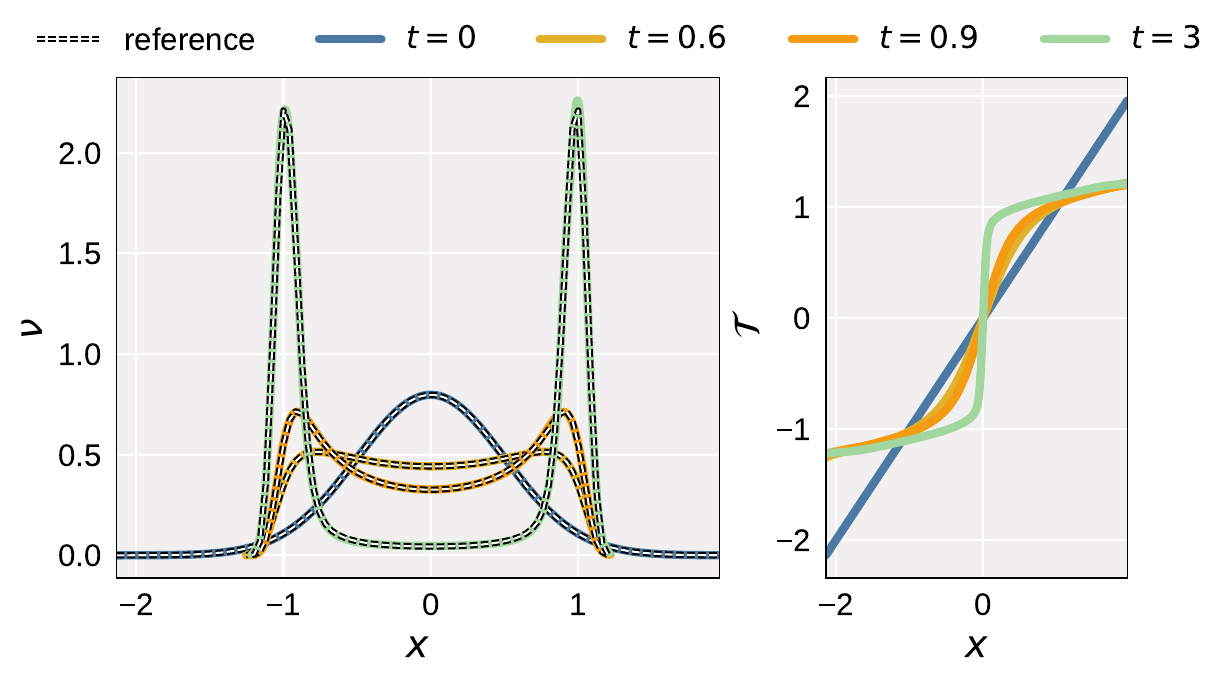}
    \caption{}
    \label{subfig:double_well_FP_1d_a}
  \end{subfigure}
  \hfill
  \begin{subfigure}[t]{0.32\textwidth}
    \centering
    \includegraphics[width=\textwidth]{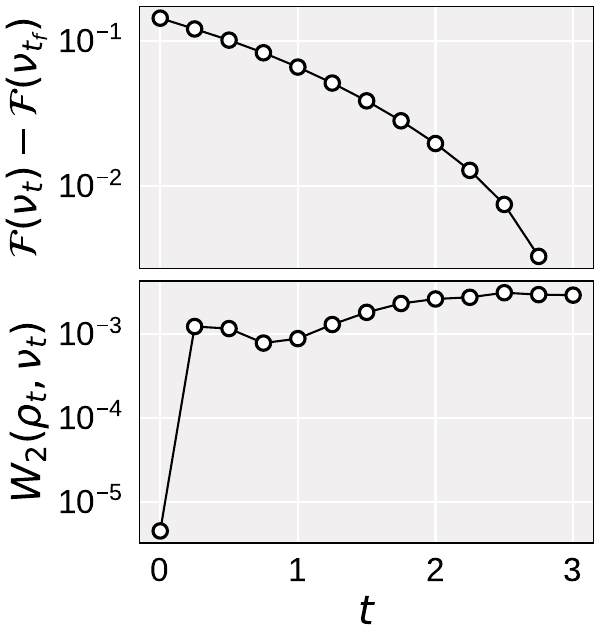}
    \caption{}
    \label{subfig:double_well_FP_1d_b}
  \end{subfigure}
  \caption{\footnotesize One-dimensional double-well Fokker--Planck equation for $k_L=0.001$
  with $n=12$, initialized with centered Gaussian $\mu_0\sim\cN(0,0.5)$.
  The density evolution is displayed in \ref{subfig:double_well_FP_1d_a}, and the
  Wasserstein error is reported together with the energy decay in \ref{subfig:double_well_FP_1d_b}.
  The results are obtained with a time-step size $\Delta t=2^{-5}$ and $N_{\text{samples}}=2^{17}$.}
  \label{fig:double_well_FP_1d}
\end{figure}

We next consider a non-convex confinement potential, namely the double-well
potential, coupled with a small amount of linear diffusion regulated by the coefficient $k_L$, with corresponding
energy
\begin{align*}
  \cF(\mu)
  =
  \int_{\bR^d}
  \left(k_L\rho(x)\log\rho(x)+\Big(\frac{1}{4}|x|^4-\frac{1}{2}|x|^2\Big)\rho(x)\right)\dx,
  \qquad
  \mu=\rho\cL^d .
\end{align*}

This test is complementary to the quadratic potential case: starting from a single
Gaussian, the density is driven away from the unstable region around the origin
and relaxes toward the two wells, producing two almost separated peaks in one
dimension (see \Cref{fig:double_well_FP_1d}). In this regime the transport map
stretches the central region rather than compressing it. Consequently, the
Jacobian does not approach the loss-of-invertibility regime observed in the
Gaussian-mixture experiment for the linear Fokker-Planck, and the push-forward formula remains well behaved
throughout the computation.

\begin{figure}[tbp!]
  \centering
  \includegraphics[width=\textwidth]{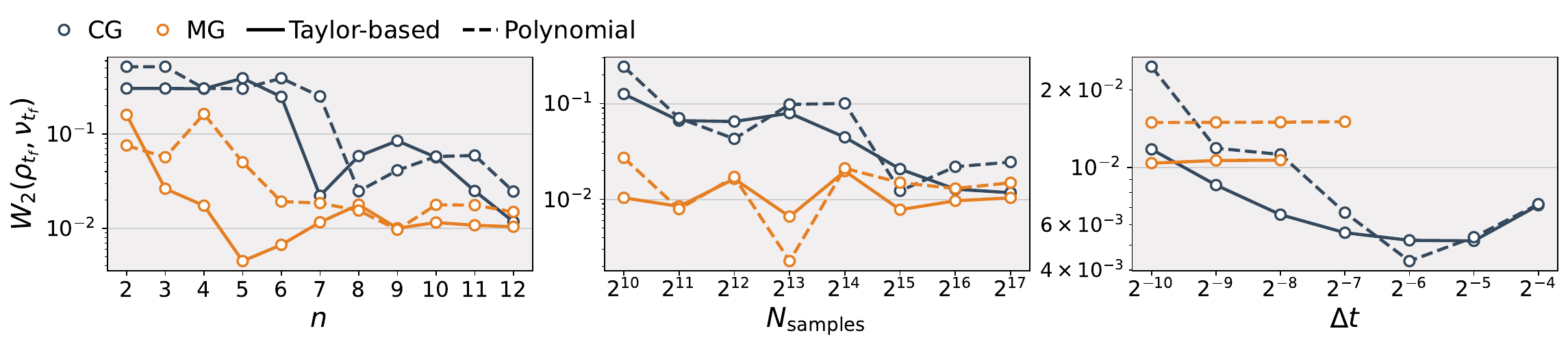}
  \caption{\footnotesize Error comparison for the one-dimensional
    double-well Fokker--Planck equation initialized both with a centered Gaussian $\mu_0\sim\cN(0,0.5)$ (CG) and multi-Gaussian $\mu_0\sim \frac{1}{2}\cN(-3,0.3)+\frac{1}{2}\cN(3,0.5)$ (MG).
    The plots report the Wasserstein error with respect to a finite element reference solution, displayed for different values of the number of basis functions, the number of samples, and the time-step size used in the method.}
  \label{fig:double_well_error_comparison}
\end{figure}

\begin{figure}[htb!]
  \centering
  \begin{subfigure}[t]{0.62\textwidth}
    \centering
    \includegraphics[width=\textwidth]{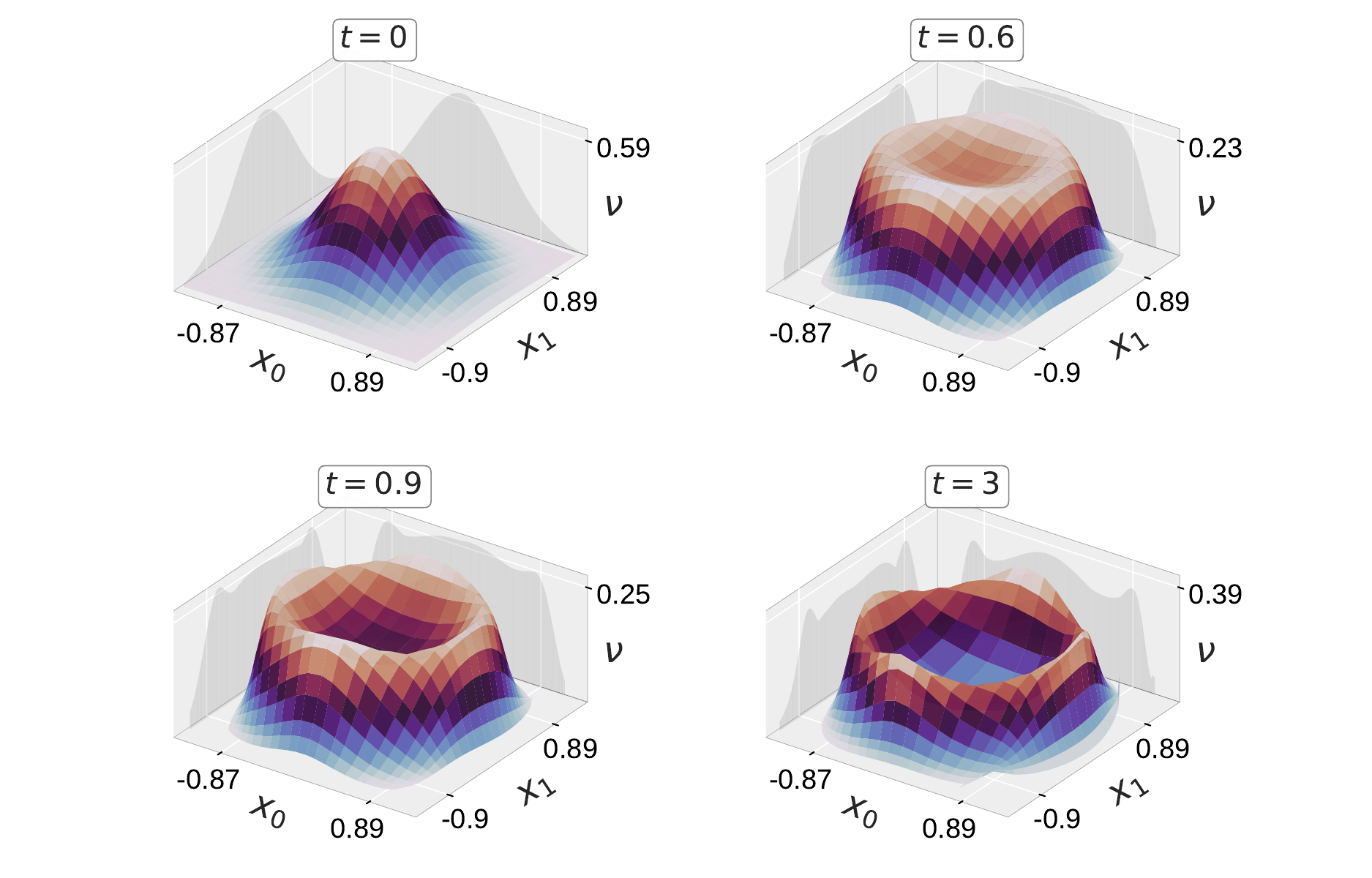}
    \caption{}
    \label{subfig:double_well_FP_2d_a}
  \end{subfigure}
  \hfill
  \begin{subfigure}[t]{0.32\textwidth}
    \centering
    \includegraphics[width=\textwidth]{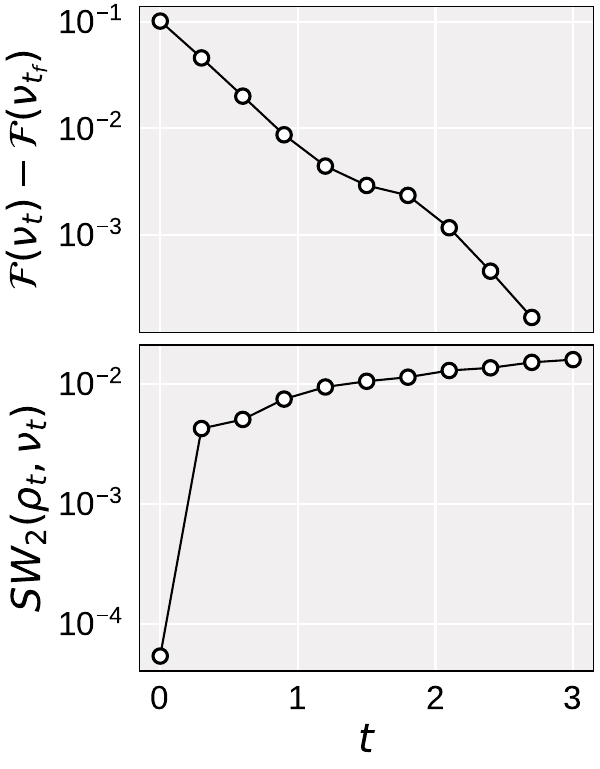}
    \caption{}
    \label{subfig:double_well_FP_2d_b}
  \end{subfigure}
  \caption{\footnotesize Two-dimensional double-well Fokker--Planck equation for $k_L=0.01$
  with $n=40$, initialized with a centered Gaussian $\mu_0\sim\cN\big((0,0),\operatorname{diag}(0.5,0.5)\big)$. The density evolution is
  displayed in \ref{subfig:double_well_FP_2d_a}, and the sliced Wasserstein
  error is reported together with the energy decay in
  \ref{subfig:double_well_FP_2d_b}. The results are obtained with a time-step
  size $\Delta t=2^{-8}$ and $N_{\text{samples}}=2^{15}$. The reference solution is
      computed with a finite element method.}
  \label{fig:double_well_FP_2d}
\end{figure}

\Cref{fig:double_well_error_comparison} shows how the
one-dimensional Wasserstein error depends on the number of basis functions, the
number of samples, and the time-step size, for both the Taylor-based background space, and a polynomial one.
The comparison shows that in this test case the
Taylor-based construction reaches a given accuracy with fewer basis functions
than the polynomial background space and remains accurate for larger time
steps. Increasing the number of samples has a milder effect on the accuracy, likely because it reduces the Monte Carlo quadrature error in the projected dynamics.

Finally, in \Cref{fig:double_well_FP_2d} we report the more demanding two-dimensional experiment of the radial double-well potential, where the set of minimizers is not a
pair of points but a curve, and the density reorganizes toward a stationary
state with radial structure. This requires a richer adaptive space than in the
linear Fokker--Planck examples, which is reflected in the larger value
$n=40$. In our implementation, the perturbation directions are taken as polynomial functions of progressively increasing degree, without any further selection criterion: all directions up to the prescribed degree are included. Consequently, in two dimensions, using 40 perturbation directions only covers polynomials up to degree five. This highlights a dimension-dependent cost of the present background-space construction. A natural direction for future work is to design selection criteria that identify the most informative perturbation directions, allowing the Taylor-based space to remain expressive while using fewer basis functions.
Nevertheless,
the method captures the qualitative shape of the density and preserves the
dissipative character of the gradient flow: the energy decreases monotonically,
and the error remains controlled over the whole time interval.

\FloatBarrier

\subsection{Nonlinear diffusion}
\begin{figure}[hbt!]
  \centering
  \begin{subfigure}[t]{0.62\textwidth}
    \centering
    \includegraphics[width=\textwidth]{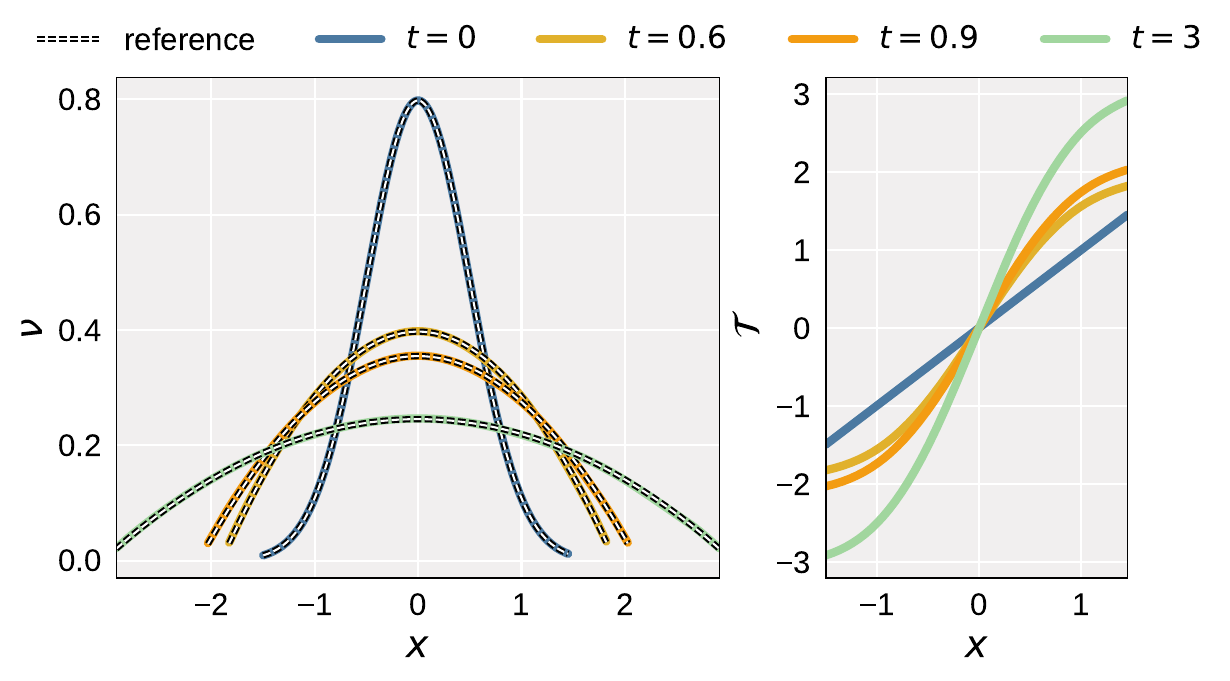}
    \caption{}
    \label{subfig:nonlinear_diffusion_1d_a}
  \end{subfigure}
  \hfill
  \begin{subfigure}[t]{0.32\textwidth}
    \centering
    \includegraphics[width=\textwidth]{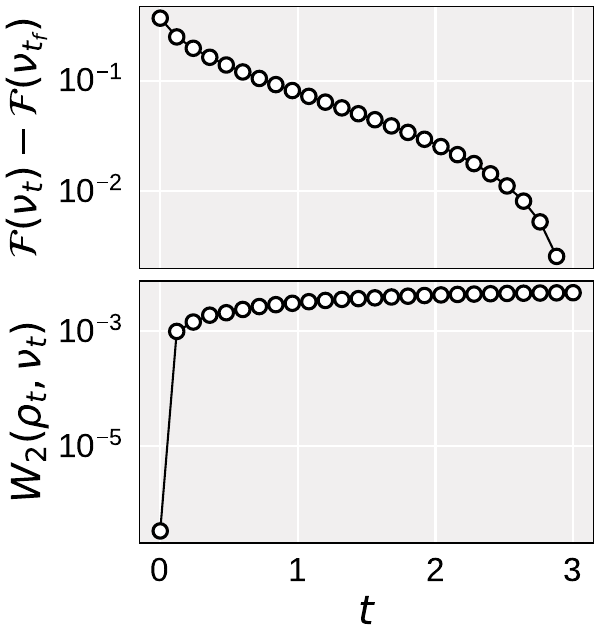}
    \caption{}
    \label{subfig:nonlinear_diffusion_1d_b}
  \end{subfigure}
  \caption{\footnotesize One-dimensional nonlinear diffusion equation with
  $m=2$ and $n=13$, initialized with a centered Gaussian
  $\mu_0\sim\cN(0.0,0.5)$. The density evolution is displayed in
  \ref{subfig:nonlinear_diffusion_1d_a}, and the Wasserstein error is
  reported together with the energy decay in
  \ref{subfig:nonlinear_diffusion_1d_b}. The results are obtained with a
  time-step size $\Delta t=2^{-6}$ and $N_{\text{samples}}=2^{15}$.}
  \label{fig:nonlinear_diffusion_1d}
\end{figure}

We finally consider the nonlinear diffusion equation, corresponding to the
Wasserstein gradient flow of the internal energy
\begin{align*}
  \cF_m(\mu)
  =
  \frac{1}{m-1}\int_{\bR^d}\rho(x)^m\dx,
  \qquad
  \mu=\rho\cL^d,\quad m>1 .
\end{align*}
This case differs from the previous Fokker--Planck tests because the
diffusivity depends on the density itself. Regions of high density spread
faster, while low-density regions evolve more slowly, producing the typical
finite-speed expansion of porous-medium dynamics.

\begin{figure}[htb]
  \centering
  \includegraphics[width=.9\textwidth]{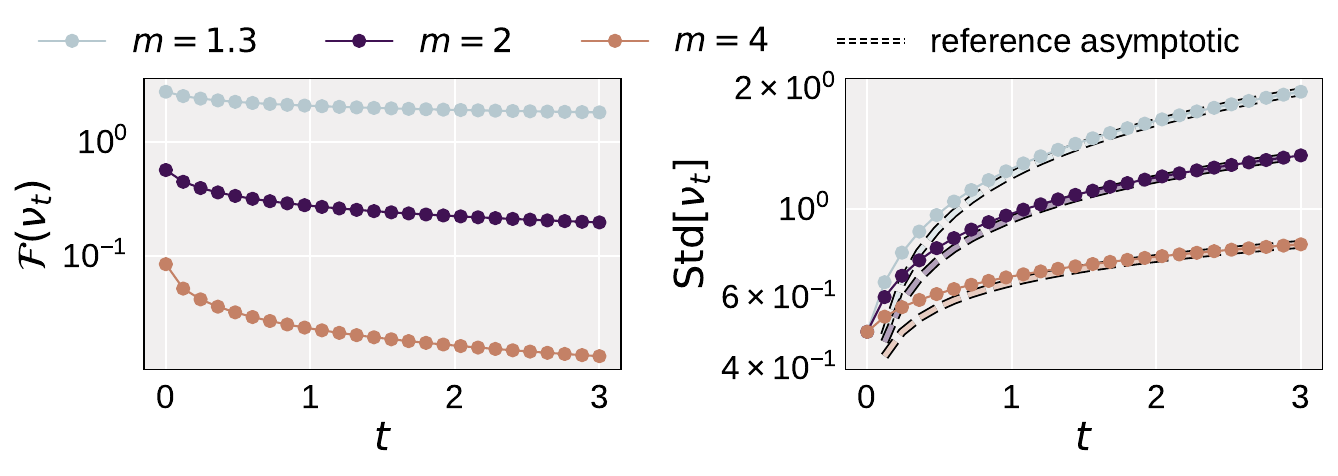}
  \caption{\footnotesize Energy decay (on the left) and spreading rate (on the right) for the one-dimensional nonlinear
  diffusion equation for different values of $m$ with $n=13$, $\Delta t=2^{-6}$ and $N_{\text{samples}}=2^{15}$, initialized with a centered Gaussian $\mu_0\sim\cN(0.0,0.5)$.
      The standard deviation of
      the numerical solution is compared with the expected self-similar scaling
      law.}
  \label{fig:nonlinear_diffusion_spreading}
\end{figure}

\Cref{fig:nonlinear_diffusion_1d} reports the one-dimensional case with
$m=2$. The method captures the spreading of the density, the Wasserstein error
with respect to the finite element reference remains controlled, and the energy
decreases along the whole time interval.
The associated transport map reflects this mechanism. Compared with the
Fokker--Planck tests, where the map also accounts for drift induced by a
confinement potential, here the dominant contribution is an approximately
symmetric expansion of the initial density. The deformation is stronger in the
region where the density is larger and weaker in the tails, in agreement with
the density-dependent diffusion characteristic of porous-medium flows.

In \Cref{fig:nonlinear_diffusion_spreading}
we compare several values of $m$ through the standard deviation of the solution.
The observed growth follows the expected self-similar law
$\operatorname{Std}(\rho_t)\sim C t^{1/(d(m-1)+2)}$ with $d=1$, confirming that
the method reproduces the correct large-scale spreading rate.

\subsection{Interaction potential}

\begin{figure}[htb]
  \centering
  \begin{subfigure}[t]{0.62\textwidth}
    \centering
    \includegraphics[width=\textwidth]{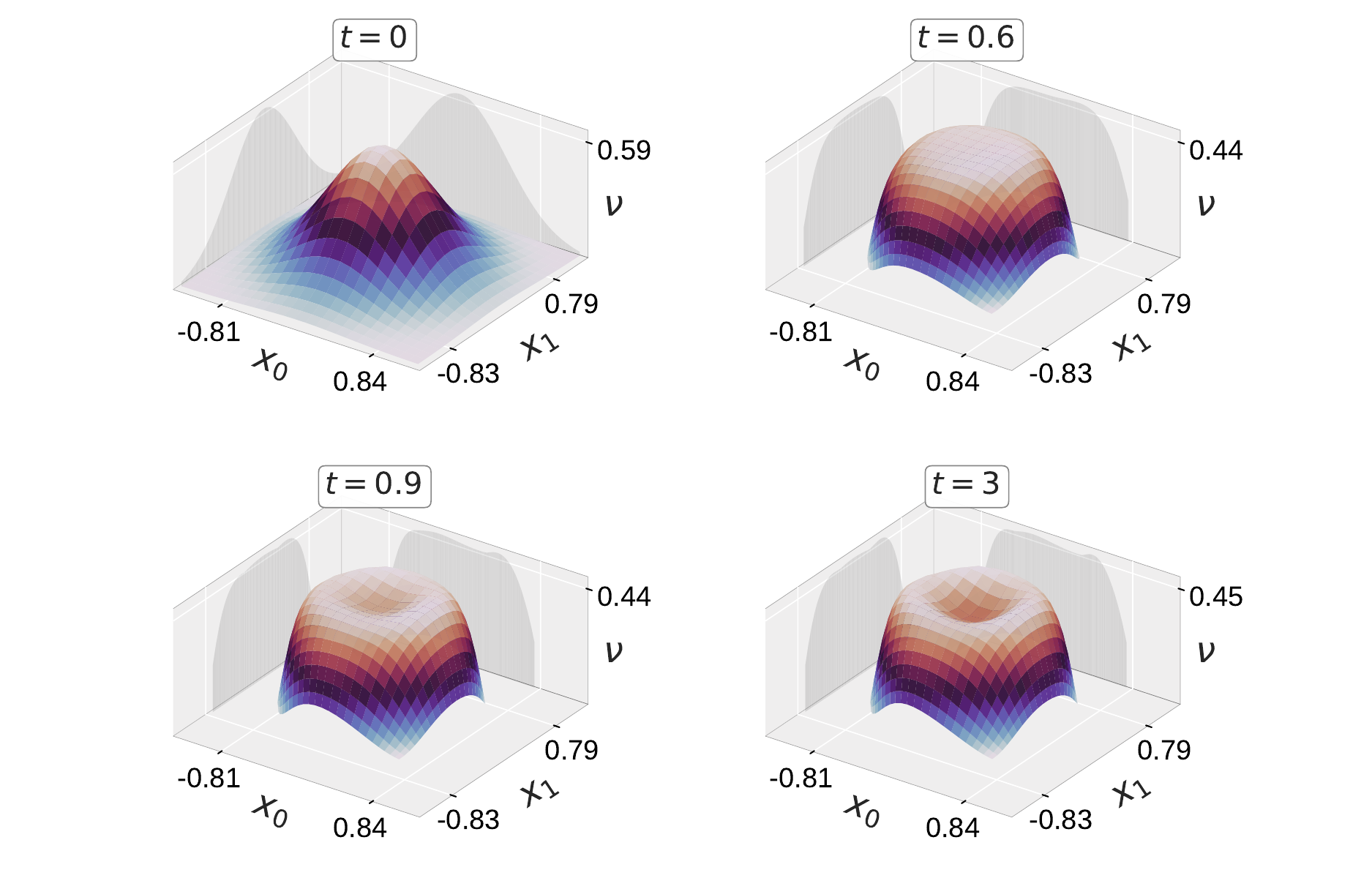}
    \caption{}
    \label{subfig:interaction_potential_2d_a}
  \end{subfigure}
  \hfill
  \begin{subfigure}[t]{0.32\textwidth}
    \centering
    \includegraphics[width=\textwidth]{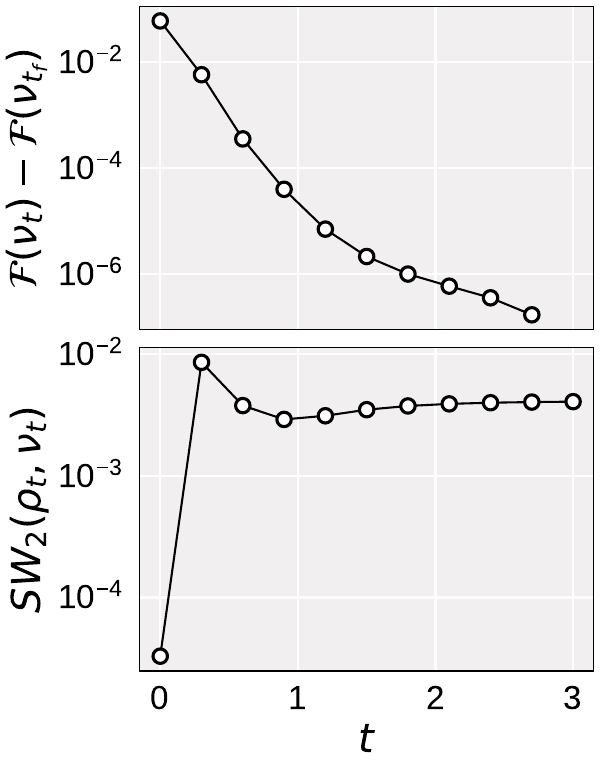}
    \caption{}
    \label{subfig:interaction_potential_2d_b}
  \end{subfigure}
  \caption{\footnotesize Two-dimensional interaction equation with $n=30$, initialized with a centered
  Gaussian $\mu_0\sim\cN\big((0,0),\operatorname{diag}(0.5,0.5)\big)$. The
  density evolution is displayed in \ref{subfig:interaction_potential_2d_a},
  and the Wasserstein error is reported together with the energy decay in
  \ref{subfig:interaction_potential_2d_b}. The results are obtained with a
  time-step size $\Delta t=2^{-8}$ and $N_{\text{samples}}=2^{15}$. The reference
      solution is computed with a finite element method.}
  \label{fig:interaction_potential_2d}
\end{figure}

We also test a nonlocal evolution driven by an interaction potential.
In this example, the energy is
\begin{align*}
  \cF(\mu)
  =
  k_L\int_{\bR^d} \rho(x)\log\rho(x)\dx
  +
  \frac12
  \int_{\bR^d}\int_{\bR^d}
  W(x-y)\rho(x)\rho(y)\dx\,\rd y,
  \qquad
  \mu=\rho\cL^d,
\end{align*}
with the attractive--repulsive potential
\begin{align*}
  W(z)=\frac14|z|^4-\frac12|z|^2 .
\end{align*}
The interaction term favors a nonzero separation between mass particles: nearby
mass is repelled, while particles farther apart are pulled back toward the
preferred interaction scale.
This example is qualitatively different from the previous ones because the
driving force at each point depends on the full distribution of mass. Hence,
the approximation space must capture not only local deformations of the
transport map, but also collective modes generated by the nonlocal interaction
term.
Therefore, compared with the confinement examples above, the
velocity field is determined by the whole density through the
convolution term $W*\rho$.

\Cref{fig:interaction_potential_2d} shows that the method captures the
nonlocal reorganization of the density induced by the interaction force. The adaptive
space remains sufficiently expressive to follow the evolution with $n=30$, and
the energy decreases monotonically along the computed trajectory.
Also in this case, as for Figure \ref{fig:double_well_FP_2d} above, the number of basis functions notably increases with the dimension.
As discussed above, this growth is largely due to our current choice of perturbation directions, which includes all monomials up to a prescribed degree without further selection. More selective constructions, for instance based on sparse polynomial spaces or problem-adapted criteria, could reduce the size of the background space while retaining the relevant directions.

\section{Conclusions}\label{sec:conclusions}
In this work, we introduce a meshless Lagrangian dynamical approximation method for Wasserstein gradient flows. 
The method approximates the transport-map evolution by a Stiefel decoder, whose moving frame is constrained to a Taylor-based background space which can be updated at prescribed times. 
This yields a regular approximation of the WGF solution measures, while avoiding the explicit solution of optimal-transport problems at each time step and not requiring the identification of spatial meshes.

The method induces an absolutely continuous curve in Wasserstein space that
satisfies a continuity equation with projected velocity. Moreover, for
geodesically convex energies, the Wasserstein error is controlled by a
Gronwall estimate in terms of the projection residual. In particular, by
suitably choosing the update times of the Taylor-based background space, the
final-time Wasserstein error can be made arbitrarily small.

The numerical experiments on linear and nonlinear Fokker--Planck equations, porous-medium diffusion, and interaction energies confirm that the method
captures the qualitative behavior of the corresponding WGFs,
including energy dissipation, spreading rates, and nonlocal density reorganization. 
They also show that the Taylor-based background space achieve a given accuracy with fewer basis functions than a fixed polynomial background space.

The main limitation of the present construction is the growth of the background
space dimension in higher dimensions. This is no longer a mesh-induced curse of
dimensionality, but rather an approximation-space bottleneck: the number of
Taylor directions may become large when the perturbation directions are not chosen wisely. 
Future work will therefore focus on sparser and more adaptive choices of the
perturbation directions in the Taylor-based construction for high-dimensional test cases, also exploiting the
possible update of the Taylor-based background space, with the aim of retaining the meshless
character of the method while improving its scalability.
Moreover, we aim to extend the method into a reduced-order modeling framework
for parametric WGFs, where the solution manifold depends
not only on time, but also on physical parameters and initial conditions.

{
	\small
	\bibliographystyle{unsrt}
	\bibliography{references}
}

\FloatBarrier
\clearpage
\appendix
\section{Convexity in the Wasserstein space}
\label{app:wasserstein-convexity}

In this appendix, we first recall the definition of Wasserstein geodesics in
\Cref{def:wasserstein-geodesic}, and then use this geodesic structure to state
the notion of $\lambda$-convexity for $\Wtwospace[\bR^d]$ in \Cref{def:lambda_convexity}.
We then derive the corresponding monotonicity estimates for the Wasserstein
subdifferential, and finally discuss the role of $\lambda$-convexity in the WGF setting.

\begin{definition}[Wasserstein geodesic]\label{def:wasserstein-geodesic}
  Let $\mu_0,\mu_1\in\Ptwospace[\bR^d]$ and let
  $\pi\in\Pi\prt{\mu_0,\mu_1}$ be an optimal transport plan. The curve
  $\prt{\mu_s}_{s\in[0,1]}$ defined by
  \begin{align*}
    \mu_s
    \coloneqq
    \pushfwd{\prt{\prt{1-s}\pi_1+s\pi_2}}{\pi},
    \qquad s\in[0,1],
  \end{align*}
  is a Wasserstein geodesic connecting $\mu_0$ and $\mu_1$.
  If $\mu_0$ is absolutely continuous and the optimal plan is induced by the
  optimal map $\Toptmap{\mu_0}{\mu_1}$, then this geodesic can be written as
  \begin{align*}
    \mu_s
    =
    \pushfwd{\prt{\prt{1-s}\id+s\Toptmap{\mu_0}{\mu_1}}}{\mu_0},\qquad s\in[0,1].
  \end{align*}
\end{definition}

\begin{definition}[$\lambda$-convexity along geodesics and generalized geodesics]\label{def:lambda_convexity}
  Let $\cF:\Ptwospace[\bR^d]\to(-\infty,+\infty]$.
  We say that $\cF$ is $\lambda$-\textit{convex along geodesics} if, for every
  $\mu_0,\mu_1\in\dom\prt{\cF}$ and every Wasserstein geodesic
  $\prt{\mu_s}_{s\in[0,1]}$ connecting them,
  \begin{equation}\label{eq:geodesic-convexity}
    \cF\prt{\mu_s}
    \le
    \prt{1-s}\cF\prt{\mu_0}+s\cF\prt{\mu_1}
    -\frac{\lambda}{2}s\prt{1-s}\Wtwometric[\mu_0][\mu_1][2],
    \qquad s\in [0,1].
  \end{equation}
  We say that $\cF$ is $\lambda$-\textit{convex along generalized geodesics}
  if, for every base measure $\omega\in\Ptwospace[\bR^d]$, every pair
  $\mu_0,\mu_1\in\dom\prt{\cF}$, and every pair of admissible transport maps
  $\Tmap{\omega}{\mu_i}\in L^2\prt{\omega;\bR^d}$ satisfying
  $\pushfwd{\Tmap{\omega}{\mu_i}}{\omega}=\mu_i$ for $i=0,1$, the curve
  $\mu_s
    =
    \pushfwd{\prt{\prt{1-s}\Tmap{\omega}{\mu_0}+s\Tmap{\omega}{\mu_1}}}{\omega}$ for $s\in[0,1]$
  satisfies
  \begin{equation}\label{eq:generalized-geodesic-convexity}
    \cF\prt{\mu_s}
    \le
    \prt{1-s}\cF\prt{\mu_0}+s\cF\prt{\mu_1}
    -\frac{\lambda}{2}s\prt{1-s}
    \int_{\bR^d}
    \normRd{\Tmap{\omega}{\mu_0}\prt{x}-\Tmap{\omega}{\mu_1}\prt{x}}^2
    \,\d\omega\prt{x}.
  \end{equation}
\end{definition}

Geodesic convexity induces the monotonicity of the Wasserstein subdifferential.
Indeed, let $\xi_0\in\partial\cF\prt{\mu_0}$, $\xi_1\in\partial\cF\prt{\mu_1}$, and assume that the
optimal transport maps between $\mu_0$ and $\mu_1$ are well defined.
Consider the geodesic connecting $\mu_0$ and $\mu_1$. From
\Cref{eq:geodesic-convexity}, the following holds:
\begin{align*}
  \frac{\cF\prt{\mu_s}-\cF\prt{\mu_0}}{s}
  \le
  \cF\prt{\mu_1}-\cF\prt{\mu_0}
  -
  \frac{\lambda}{2}\prt{1-s}\Wtwometric[\mu_0][\mu_1][2]\quad \forall s\in(0,1].
\end{align*}
Passing to the limit as $s\downarrow0$ and using the subdifferential definition in \Cref{eq:subdifferential-definition}, we get
\begin{align*}
  \cF\prt{\mu_1}-\cF\prt{\mu_0}
  &\ge
  \int_{\bR^d}
  \innerRd{\xi_0, \Toptmap{\mu_0}{\mu_1}-\id}
  \,\d\mu_0\prt{x}
  +
  \frac{\lambda}{2}\Wtwometric[\mu_0][\mu_1][2],\\
  \cF\prt{\mu_0}-\cF\prt{\mu_1}
  &\ge
  \int_{\bR^d}
  \innerRd{\xi_1, \Toptmap{\mu_1}{\mu_0}-\id}
  \,\d\mu_1\prt{x}
  +
  \frac{\lambda}{2}\Wtwometric[\mu_0][\mu_1][2].
\end{align*}
Then, changing variables in the second integral and summing the two endpoint inequalities yields the monotonicity of the
Wasserstein subdifferential:
\begin{equation}\label{eq:geodesic-subdifferential-monotonicity}
  \int_{\bR^d}
  \innerRd{
    \xi_1\circ\Toptmap{\mu_0}{\mu_1}-\xi_0,
    \Toptmap{\mu_0}{\mu_1}-\id
  }
  \,\d\mu_0\prt{x}
  \ge
  \lambda \Wtwometric[\mu_0][\mu_1][2].
\end{equation}

The same computation gives the monotonicity estimate for $\lambda$-convex functionals along
generalized geodesics.
Indeed, applying
\Cref{eq:generalized-geodesic-convexity} at $\mu_0$ and
$\mu_1$ and summing the two inequalities yields
\begin{equation}\label{eq:wasserstein-subdifferential-monotonicity}
  \int_{\bR^d}
  \innerRd{
    \xi_1\circ \Tmap{\omega}{\mu_1}\prt{x}
    -\xi_0\circ \Tmap{\omega}{\mu_0}\prt{x},
    \Tmap{\omega}{\mu_1}\prt{x}-\Tmap{\omega}{\mu_0}\prt{x}
  }
  \,\d\omega\prt{x}
  \ge
  \lambda
  \int_{\bR^d} \normRd{\Tmap{\omega}{\mu_1}\prt{x}-\Tmap{\omega}{\mu_0}\prt{x}}^2\,\d\omega\prt{x}
\end{equation}
for every pair of admissible transport maps $\Tmap{\omega}{\mu_0}$ and $\Tmap{\omega}{\mu_1}$.

We now explain how the $\lambda$-convexity of the energy functional impacts the
corresponding WGF. The first quantity to consider is the
metric derivative, recalled in \Cref{def:metric-derivative}, which measures the
instantaneous speed of a curve in $\Wtwospace[\bR^d]$.
\begin{definition}[Metric derivative]\label{def:metric-derivative}
  Let $\prt{\mu_t}_{t\in[0,t_f]}$ be a curve in
  $\prt{\Ptwospace[\bR^d],\Wtwometric}$. Its metric derivative, when it exists, is defined by
  \begin{align*}
    \abs{\mu'}\prt{t}
    \coloneqq
    \lim_{h\to0}
    \frac{\Wtwometric[\mu_{t+h}][\mu_t]}{\abs{h}}.
  \end{align*}
\end{definition}
We point out that, if $\prt{\mu_t}_{t\in[0,t_f]}$ is a WGF with
velocity $v_{\mu_t}=-\partial^\circ\cF\prt{\mu_t}$, then the following characterization identity holds (see Theorem~8.3.1 in \cite{AGS2008}):
\begin{align}\label{eq:metric-derivative-identity}
  \abs{\mu'}\prt{t}=\norm{v_{\mu_t}}_{\Ltwow{\mu_t}}
  \qquad\text{for a.e. }t\in[0,t_f].
\end{align}
Thus, the metric derivative is the natural way to measure the speed of the
flow. When $\cF$ is $\lambda$-convex, the associated Wasserstein gradient-flow
semigroup is $\lambda$-contractive (see Theorem~11.1.4 in \cite{AGS2008}),
which yields the exponential decay of this speed:
\begin{align}\label{eq:metric-derivative-decay}
  \abs{\mu'}\prt{t}
  \le
  e^{-\lambda\prt{t-s}}\abs{\mu'}\prt{s},
  \qquad 0\le s\le t\le t_f .
\end{align}

The same convexity assumption also gives a decay estimate at the level of the
energy. If $\lambda>0$ and $\cF$ admits a minimizer $\mu_\infty$, then the
energy gap decays exponentially along the flow (see Theorem~2.4.14 in
\cite{AGS2008}):
\begin{align}\label{eq:energy_decay}
  \cF\prt{\mu_t}-\cF\prt{\mu_\infty}
  \le
  e^{-2\lambda t}\prt{\cF\prt{\mu_0}-\cF\prt{\mu_\infty}},
  \qquad 0\leq t\leq t_f.
\end{align}

\section{Stiefel manifold}
\label{sec:stiefel-notions}
We denote by $\St\prt{n,L^2_{\mu_0}}$ the set of $n$-frames with components in $L^2_{\mu_0}$ that are orthonormal with respect to this inner product, namely
\begin{equation}
  \St\prt{n,L^2_{\mu_0}}
  \coloneqq
  \set[\big]{
    \stiefelframe{v}=\{v_i\}_{i=1}^n \in \prt{L^2_{\mu_0}}^n
    \cond
    \inner{v_i, v_j}_{L^2_{\mu_0}}=\delta_{i,j},\;\; 1\leq i,j\leq n
  }.
\end{equation}
The tangent space at $\stiefelframe{v} \in \St\prt{n,L^2_{\mu_0}}$ consists of all $n$-frames
$\stiefelframe{w}=\{w_i\}_{i=1}^n\in\prt{L^2_{\mu_0}}^n$ whose cross-Gram matrix with $\stiefelframe{v}$ is skew-symmetric, i.e.
\begin{equation}\label{eq:tangent_stiefel}
  \Tan_{\stiefelframe{v}} \St\prt{n,L^2_{\mu_0}}
  =
  \set[\big]{
    \stiefelframe{w}=\{w_i\}_{i=1}^n \in \prt{L^2_{\mu_0}}^n
    \cond
    \inner{w_i, v_j}_{L^2_{\mu_0}}+ \inner{v_i, w_j}_{L^2_{\mu_0}} =0,\;\; 1\leq i,j\leq n
  }.
\end{equation}
We endow $\St\prt{n,L^2_{\mu_0}}$ with the metric induced by the $L^2_{\mu_0}$ inner product, namely
\begin{equation}
  g_{\stiefelframe{v}}\prt{\stiefelframe{w},\tilde{\stiefelframe{w}}}
  \coloneqq
  \sum_{i=1}^n
  \inner{w_i, \tilde{w}_i}_{L^2_{\mu_0}},
  \qquad
  \forall \stiefelframe{w},\tilde{\stiefelframe{w}}\in \Tan_{\stiefelframe{v}} \St\prt{n,L^2_{\mu_0}}.
\end{equation}
We can further decompose the tangent space $\Tan_{\stiefelframe{v}} \St\prt{n,L^2_{\mu_0}}$ as
\begin{equation}
  \Tan_{\stiefelframe{v}} \St\prt{n,L^2_{\mu_0}}
  = \rH^\St_{\stiefelframe{v}} \oplus \rV^\St_{\stiefelframe{v}},
\end{equation}
where we distinguish between the horizontal subspace
\begin{align}\label{eq:horizontal_tangent_stiefel}
  \rH_{\stiefelframe{v}}^\St
  \coloneqq
  \set[\big]{
    \stiefelframe{w} \in \prt{L^2_{\mu_0}}^n
    \cond
    w_i \in \vspan\set{\stiefelframe{v}}^{\perp_{L^2_{\mu_0}}}, \; 1\leq i \leq n
  }
\end{align}
and its vertical subspace
\begin{align}\label{eq:vertical_tangent_stiefel}
  \rV_{\stiefelframe{v}}^\St
  \coloneqq
  \set[\big]{
    \stiefelframe{w} \in \prt{L^2_{\mu_0}}^n
    \cond
    w_i \in \vspan\set{\stiefelframe{v}} \text{ and }
    \inner{w_i, v_j}_{L^2_{\mu_0}} + \inner{w_j, v_i}_{L^2_{\mu_0}} = 0,\quad 1\leq i,j \leq n
  }.
\end{align}

To time-integrate the moving basis functions of our Stiefel decoder while preserving their geometric structure, we introduce the exponential map on $\St\prt{n,L^2_{\mu_0}}$ at a point $\stiefelframe{v}\in\St\prt{n,L^2_{\mu_0}}$.
It is defined as
\begin{align}
  \label{eq:exp-stiefel}
  \Exp_{\stiefelframe{v}}: \Tan_{\stiefelframe{v}} \St\prt{n,L^2_{\mu_0}} & \to \St\prt{n,L^2_{\mu_0}}, \\
  \stiefelframe{w}                                                       & \mapsto \gamma_{\stiefelframe{w}}\prt{1},
\end{align}
where $\gamma_{\stiefelframe{w}}: [0,1]\to \St\prt{n,L^2_{\mu_0}}$ is the constant speed geodesic with initial conditions $\gamma_{\stiefelframe{w}}\prt{0}=\stiefelframe{v}$ and velocity $\dot \gamma_{\stiefelframe{w}}\prt{0} = \stiefelframe{w}$.
Computing the exponential map is not a trivial operation in practice: given an element $\stiefelframe{w}\in\Tan_{\stiefelframe{v}}\St\prt{n,L^2_{\mu_0}}$, it holds that (see Proposition $1$ in \cite{harms2012geodesics})
\begin{align}
  \label{eq:exp-map}
  \Exp_{\stiefelframe{v}}\prt{\stiefelframe{w}}
  =
  \begin{bmatrix}
    v_1 & \cdots & v_n & w_1 & \cdots & w_n
  \end{bmatrix}
  \exp\prt[\Bigg]{
    \begin{bmatrix}
      G_{\stiefelframe{v}, \stiefelframe{w}} & -G_{\stiefelframe{w}, \stiefelframe{w}} \\
      \Id_n                                 & G_{\stiefelframe{v}, \stiefelframe{w}}
    \end{bmatrix}
  }
  \begin{bmatrix}
    \Id_n \\
    0_n
  \end{bmatrix}
  e^{-G_{\stiefelframe{v}, \stiefelframe{w}}},
\end{align}
where $G_{\stiefelframe{v}, \stiefelframe{w}}, G_{\stiefelframe{w}, \stiefelframe{w}} \in \bR^{n\times n}$ are defined as
\begin{align*}
  G_{\stiefelframe{v}, \stiefelframe{w},ij}
  =
  \inner{v_i, w_j}_{L^2_{\mu_0}},
  \quad \text{and}\quad
  G_{\stiefelframe{w}, \stiefelframe{w},ij}
  =
  \inner{w_i, w_j}_{L^2_{\mu_0}},
  \quad 1\leq i,j\leq n.
\end{align*}
In practice, for general vectors $\stiefelframe{w}\in\Tan_{\stiefelframe{v}}\St\prt{n,L^2_{\mu_0}}$, the exponential of the block matrix and of the $G_{\stiefelframe{v}, \stiefelframe{w}}$ matrix are usually numerically approximated using a power series expansion truncated at some order.
However, if we consider a horizontal frame $\stiefelframe{w}\in\rH_{\stiefelframe{v}}^\St$ with pairwise orthogonal components, then $G_{\stiefelframe{v}, \stiefelframe{w}}=0_n$ and $G_{\stiefelframe{w}, \stiefelframe{w}}$ becomes diagonal.
Thus, the exponential map reduces to the following componentwise closed formula for each $1\leq i\leq n$:
\begin{align}\label{eq:exp_closed_formula}
  \Exp_{\stiefelframe{v}}\prt{\stiefelframe{w}}_i
  =
  \begin{cases}
    \cos\prt{\norm{w_i}_{L^2_{\mu_0}}}v_i+\dfrac{\sin\prt{\norm{w_i}_{L^2_{\mu_0}}}}{\norm{w_i}_{L^2_{\mu_0}}}w_i,
     & \norm{w_i}_{L^2_{\mu_0}}\neq 0, \\
    v_i,
     & \norm{w_i}_{L^2_{\mu_0}}= 0.
  \end{cases}
\end{align}

Finally, we introduce the finite-dimensional Stiefel submanifold associated with a finite-dimensional linear subspace $\bW_m\subset L^2_{\mu_0}$ as follows:
\begin{equation}
  \St\prt{n,\bW_m;L^2_{\mu_0}}
  \coloneqq
  \set[\big]{
    \stiefelframe{v}=\{v_i\}_{i=1}^n\in \bW_m
    \cond
    \inner{v_i,v_j}_{L^2_{\mu_0}}=\delta_{i,j},
    \;\; 1\leq i,j\leq n
  }
  \subset \St\prt{n,L^2_{\mu_0}}.
\end{equation}
The metric is the one induced by the $L^2_{\mu_0}$ inner product, while the tangent space, its horizontal--vertical decomposition, and the exponential map are obtained by restricting the corresponding objects defined above on $\St\prt{n,L^2_{\mu_0}}$ to frames and tangent vectors with components in $\bW_m$.
In particular, for every $\stiefelframe{v}\in\St\prt{n,\bW_m;L^2_{\mu_0}}$, the horizontal tangent directions are taken in
$\bL_{\stiefelframe{v}}\coloneqq\bW_m\cap \vspan\set{\stiefelframe{v}}^{\perp_{L^2_{\mu_0}}}$,
which is now a finite-dimensional space of dimension $k=m-n$.

\section{Proof of \Cref{thm:params-evolution}}
\label{sec:params-evolution}


We momentarily freeze time and consider the abstract least-squares problem underlying \Cref{eq:least-squares_g}.
Let $\bW_m\subset L^2_{\mu_0}$ be a finite-dimensional linear background space with $\dim\prt{\bW_m}=m\geq n$, let $\bc\in\bR^n$, and $\stiefelframe{v}\in\St\prt{n,\bW_m;L^2_{\mu_0}}$.
Set $\theta=\prt{\bc,\stiefelframe{v}}$, and let $\xi\in L^2_{\mu_0}$ be a target element.
Then the abstract minimization problem is:
\begin{equation}
  \label{eq:min-pb-generic}
  \min_{\substack{\prt{\bh, \stiefelframe{w}} \,\in\, \\ \Tan_\theta \Theta =\bR^n\times \Tan_{\stiefelframe{v}}\St\prt{n,\bW_m;L^2_{\mu_0}}}} \cJ\prt{\bh, \stiefelframe{w}}, \quad \text{with}\quad \cJ\prt{\bh, \stiefelframe{w}}\coloneqq
  \frac 1 2 \norm[\Big]{\sum_{i=1}^n c_iw_i+h_iv_i - \xi}_{L^2_{\mu_0}}^2.
\end{equation}
We proceed in two steps.
First, in \Cref{lem:horizontal-suffices}, we show that \Cref{eq:least-squares_g} is equivalent to the problem obtained by restricting $\stiefelframe{w}$ to the horizontal tangent space $\rH_{\stiefelframe{v}}^{\St}\subset \Tan_{\stiefelframe{v}}\St\prt{n,\bW_m;L^2_{\mu_0}}$.
Then, in \Cref{lem:horizontal_characterization} we characterize the minimizers of this equivalent problem.
\begin{lemma}[Horizontal vectors are minimizers]
  \label{lem:horizontal-suffices}
  For any $\theta=\prt{\bc, \stiefelframe{v}} \in \Theta$, and any $\xi\in L^2_{\mu_0}$,
  \begin{align}
    \label{eq:horizontal-vecs-are-minimizers}
    \min_{\substack{\prt{\bh, \stiefelframe{w}} \,\in \bR^n\times \Tan_{\stiefelframe{v}}\St\prt{n,\bW_m;L^2_{\mu_0}}}} \cJ\prt{\bh,\stiefelframe{w}}
    =
    \min_{\substack{\prt{\bm{\tilde{h}}, \tilde{\stiefelframe{w}}} \,\in \bR^n\times \rH_{\stiefelframe{v}}^{\St}}} \cJ\prt{\tilde \bh, \tilde{\stiefelframe{w}}}.
  \end{align}
\end{lemma}
\begin{proof}
  Let $\theta=\prt{\bc, \stiefelframe{v}} \in \Theta$. Since $\bR^n \times \rH_{\stiefelframe{v}}^{\St} \subset \Tan_\theta \Theta$, it holds that
  \begin{align*}
    \min_{\substack{\prt{\bh, \stiefelframe{w}} \,\in\bR^n\times \Tan_{\stiefelframe{v}}\St\prt{n,\bW_m;L^2_{\mu_0}}}} \cJ\prt{\bh, \stiefelframe{w}}
    \leq
    \min_{\substack{\prt{\bm{\tilde{h}}, \tilde{\stiefelframe{w}}} \,\in\bR^n\times \rH_{\stiefelframe{v}}^{\St}}}
    \cJ\prt{\tilde \bh, \tilde{\stiefelframe{w}}}.
  \end{align*}
  We next show that equality holds. The minimum is attained since this is a finite-dimensional least-squares problem. Let $\prt{\bh^\opt,\stiefelframe{w}^\opt} \in \Tan_\theta \Theta$ be a minimizer of the left-hand side problem. We can uniquely decompose $\stiefelframe{w}^\opt$ as $\stiefelframe{w}^\opt=\stiefelframe{w}^{\rH}+\stiefelframe{w}^{\rV}$ with $\stiefelframe{w}^{\rH}\in\rH_{\stiefelframe{v}}^{\St}$ and $\stiefelframe{w}^{\rV}\in\rV_{\stiefelframe{v}}^{\St}$. Since $w_i^{\rV} \in \vspan\set{\stiefelframe{v}}$, the vertical component can be written as $w_i^{\rV}=\sum_{j=1}^n a_{ij}v_j$ with $a_{ij}=-a_{ji}=\inner{w_i^{\rV}, v_j}_{L^2_{\mu_0}}$ for all $1\leq i\leq n$. Consequently, we can write
  \begin{align}
    \min_{\substack{\prt{\bh, \stiefelframe{w}} \,\in\bR^n\times \Tan_{\stiefelframe{v}}\St\prt{n,\bW_m;L^2_{\mu_0}}}} \cJ\prt{\bh, \stiefelframe{w}}
     & = \cJ\prt{\bh^\opt,\stiefelframe{w}^\opt}                                                                             \\
     & = \frac 1 2 \norm[\Big]{\sum_{i=1}^n c_i \prt[\big]{w^{\rH}_i+w_i^{\rV}}+h_i^\opt v_i-\xi}_{L^2_{\mu_0}}^2                          \\
     & =\frac 1 2 \norm[\Big]{\sum_{i=1}^n c_iw_i^{\rH}+v_i\prt[\Big]{h_i^\opt+\sum_{j=1}^n a_{ji}c_j}-\xi}_{L^2_{\mu_0}}^2 \\
     & = \cJ\prt{\tilde \bh, \stiefelframe{w}^{\rH}} \quad \text{with }\tilde{h}_i:=h_i^\opt+\sum_{j=1}^n a_{ji}c_j                   \\
     & \geq \min_{\prt{\bm{\tilde{h}}, \tilde{\stiefelframe{w}}} \,\in\, \bR^n\times \rH_{\stiefelframe{v}}^{\St}}
    \cJ\prt{\tilde \bh, \tilde{\stiefelframe{w}}}.
  \end{align}
  This yields the equality in \Cref{eq:horizontal-vecs-are-minimizers} on the minimal value of both problems. From the proof, it also follows that if $\prt{\bh^\opt,\stiefelframe{w}^\opt} \in \Tan_\theta \Theta$ is a minimizer of the left-hand side problem with $\stiefelframe{w}^\opt=\stiefelframe{w}^{\rH}+\stiefelframe{w}^{\rV}$, then $\prt{\tilde \bh, \stiefelframe{w}^{\rH}}$ is a minimizer of the right-hand side problem.
\end{proof}

\begin{lemma}[Characterization of the horizontal minimizers]
  \label{lem:horizontal_characterization}
  For any $\theta=\prt{\bc, \stiefelframe{v}}\in \Theta$ and $\xi\in L^2_{\mu_0}$, the minimizers $\prt{\bh^*, \stiefelframe{w}^*}\in \bR^n\times \rH_{\stiefelframe{v}}^{\St}$ of problem
  \begin{equation}
    \label{eq:min-pb-horizontal}
    \min_{\substack{\prt{\bm{h}, \stiefelframe{w}} \,\in \bR^n\times \rH_{\stiefelframe{v}}^{\St}}} \cJ\prt{\bh, \stiefelframe{w}}
  \end{equation}
  are of the form $h_i^* = \inner{ \xi, v_i}_{L^2_{\mu_0}}$ and $w_i^* =\sum_{j=1}^{k} a^*_{i,j} \psi_j$, with $i\in\{1,\dots,n\}$ and $k=m-n$.
  Moreover, $V_n\coloneqq\vspan\set{\stiefelframe{v}}$, $\{\psi_i\}_{i=1}^{k}$ is an $L^2_{\mu_0}$-orthonormal basis of $\bL_{\stiefelframe{v}}$, and $A^*=\{a_{i,j}^*\}_{\substack{1\leq i\leq n \\ 1\leq j\leq k}} \in \bR^{n\times k}$ is a minimizer of
  \begin{equation}
    \label{eq:optim-stiefel-Rnk}
    A^* \in \argmin_{A\in \bR^{n\times k}} \; \abs{A^\top \bc - \bb}_{\bR^k}^2 
  \end{equation}
  with $\bb \coloneqq \{\inner{\proj{V_n^\perp}{L^2_{\mu_0}}\prt{\xi}, \psi_i}_{L^2_{\mu_0}}\}_{1\leq i \leq k}$.
\end{lemma}

\begin{proof}
  By \Cref{lem:horizontal-suffices}, we search for minimizers $\prt{\bh^*, \stiefelframe{w}^*}\in \bR^n \times \rH_{\stiefelframe{v}}^{\St}$. Let $\prt{\bh, \stiefelframe{w}}\in \bR^n \times \rH_{\stiefelframe{v}}^{\St}$, set $V_n\coloneqq\vspan\set{\stiefelframe{v}}$, and observe that, by orthogonality, we can decompose the objective function as
  \begin{align}
    \norm[\Big]{\sum_{i=1}^n c_iw_i+h_iv_i - \xi}_{L^2_{\mu_0}}^2
     & = \norm[\Big]{\sum_{i=1}^n c_i w_i - \proj{V_n^\perp}{L^2_{\mu_0}}\prt{\xi}}_{L^2_{\mu_0}}^2
    +  \norm[\Big]{\sum_{i=1}^n h_i v_i - \proj{V_n}{L^2_{\mu_0}}\prt{\xi}}_{L^2_{\mu_0}}^2.
  \end{align}
  This yields a decoupling of the minimization problem as
  \begin{align}
    &\min_{\prt{\bh, \stiefelframe{w}}\in \bR^n \times \rH_{\stiefelframe{v}}^{\St}}
    \norm[\Big]{\sum_{i=1}^n c_iw_i+h_iv_i - \xi}_{L^2_{\mu_0}}^2
    \notag\\
    &\quad =
    \min_{\stiefelframe{w} \in \rH_{\stiefelframe{v}}^{\St}}
    \norm[\Big]{\sum_{i=1}^n c_i w_i - \proj{V_n^\perp}{L^2_{\mu_0}}\prt{\xi}}_{L^2_{\mu_0}}^2
    +
    \min_{\bh\in \bR^n}
    \norm[\Big]{\sum_{i=1}^n h_i v_i - \proj{V_n}{L^2_{\mu_0}}\prt{\xi}}_{L^2_{\mu_0}}^2.
  \end{align}
  The unique minimizer to the second optimization problem is $\bh^*$ such that $\proj{V_n}{L^2_{\mu_0}}\prt{\xi} = \sum_{i=1}^n h_i^* v_i$, which yields the stated formula for $\bh^*$.

  We now turn to the optimization problem over $\stiefelframe{w}$, and start by observing that since $\stiefelframe{w} \in \rH_{\stiefelframe{v}}^{\St}$, then we can express its components as $w_i = \sum_{j=1}^k a_{i,j} \psi_j \in \bL_{\stiefelframe{v}}$ for coefficients $a_{i,j}$ which we gather in a matrix $A\in \bR^{n\times k}$. It follows that the loss function of the corresponding minimization problem can be written as
  \begin{equation}
    \norm[\Big]{\sum_{i=1}^n c_i w_i - \proj{V_n^\perp}{L^2_{\mu_0}}\prt{\xi}}_{L^2_{\mu_0}}^2 = \abs{ A^\top \bc - \bb}_{\bR^k}^2 + \norm[\big]{\prt[\big]{\id - \proj{\bL_{\stiefelframe{v}}}{L^2_{\mu_0}}}\prt{\proj{V_n^\perp}{L^2_{\mu_0}}\prt{\xi}}}_{L^2_{\mu_0}}^2
  \end{equation}
  with $\bb$ defined in the statement of the lemma. Therefore
  \begin{align}
    \label{eq:loss-intermediate}
    \min_{\stiefelframe{w} \in \rH_{\stiefelframe{v}}^{\St}}  \norm[\Big]{\sum_{i=1}^n c_i w_i - \proj{V_n^\perp}{L^2_{\mu_0}}\prt{\xi}}_{L^2_{\mu_0}}^2
    =
    \norm[\big]{\prt[\big]{\id - \proj{\bL_{\stiefelframe{v}}}{L^2_{\mu_0}}}\prt{\proj{V_n^\perp}{L^2_{\mu_0}}\prt{\xi}}}_{L^2_{\mu_0}}^2
    +
    \min_{A\in \bR^{n\times k}} \abs{ A^\top \bc - \bb}_{\bR^k}^2.
  \end{align}
  The set of minimizers of the right-hand side problem is
  \begin{equation}
    \label{eq:A-star-abstract}
    A^* =
    \begin{cases}
       \text{any }A \in \bR^{n\times k}, & \text{if }\bc=0,                                                   \\
       \frac{1}{\abs{\bc}^2} \bc \bb^\top + Z, \text{ with }Z \in \bR^{n\times k}\st Z^\top \bc =0, & \text{if }\bc\neq 0,
    \end{cases}
  \end{equation}
  and we observe that the minimum of \Cref{eq:loss-intermediate} is equal to $\norm[\big]{\prt[\big]{\id - \proj{\bL_{\stiefelframe{v}}}{L^2_{\mu_0}}}\prt{\proj{V_n^\perp}{L^2_{\mu_0}}\prt{\xi}}}_{L^2_{\mu_0}}^2$ if $\bc\neq 0$ and to $\abs{\bb}^2+\norm[\big]{\prt[\big]{\id - \proj{\bL_{\stiefelframe{v}}}{L^2_{\mu_0}}}\prt{\proj{V_n^\perp}{L^2_{\mu_0}}\prt{\xi}}}_{L^2_{\mu_0}}^2 = \norm{\proj{V_n^\perp}{L^2_{\mu_0}}\prt{\xi}}_{L^2_{\mu_0}}^2$ if $\bc=0$.
  For one of the admissible $A^*$, we assemble the corresponding $\stiefelframe{w}^*$, which yields the stated formula for $\stiefelframe{w}^*$.
\end{proof}

Then, the proof of \Cref{thm:params-evolution} follows naturally by applying the previous lemmas to the least-squares problem
\Cref{eq:least-squares_g}. Indeed, fixing a time $t$, it is obtained from our abstract form in \Cref{eq:min-pb-generic} by considering
$\theta=\prt{\bc_t,\stiefelframe{v}_t}$ and $\xi=v_{\nu_t}\circ\dec_t$.
Thus, \Cref{lem:horizontal-suffices} and
\Cref{lem:horizontal_characterization}  prove
the parameter evolution stated in \Cref{thm:params-evolution}.

\section{Auxiliary results}

\begin{lemma}[Two-curve derivative estimate]\label{lemma:dt_W2}
  Let $\mu_t,\nu_t\in AC\prt{[0,t_f],\Ptwospace[\bR^d]}$ solve the continuity equations
  with velocity fields $v_{\mu_t}$ and $w_{\nu_t}$, respectively. Assume moreover that
  $\mu_t$ and $\nu_t$ are absolutely continuous with respect to $\Leb{d}$ for
  a.e.\ $t\in[0,t_f]$.
  Then, for $\mathcal L^1$-a.e.\ $t\in[0,t_f]$,
  \begin{align}\label{eq:two-curve-ineq}
    \frac{\d}{\dt}\frac12 \Wtwometric[\mu_t][\nu_t][2]
    \leq
    \int_{\bR^d}
    \inner[\big]{
      v_{\mu_t}\prt{x}-w_{\nu_t}\circ\Toptmap{\mu_t}{\nu_t}\prt{x},
      x-\Toptmap{\mu_t}{\nu_t}\prt{x}
    }
    \d\mu_t\prt{x}.
  \end{align}
\end{lemma}

\begin{proof}
  By \cite[Proposition 8.4.7]{AGS2008}, since $\mu_t$ is an absolutely
  continuous curve with velocity $v_{\mu_t}$, for every fixed
  $\nu\in\Ptwospace[\bR^d]$ it holds for $\mathcal L^1$-a.e.\ $t\in[0,t_f]$
  that
  \begin{align}\label{eq:dt-W2-first-curve}
    \frac{\d}{\dt}\frac12 \Wtwometric[\mu_t][\nu][2]
    =
    \int_{\bR^d}
    \inner{v_{\mu_t}\prt{x},x-\Toptmap{\mu_t}{\nu}\prt{x}}
    \d\mu_t\prt{x}.
  \end{align}
  Similarly, for every fixed $\mu\in\Ptwospace[\bR^d]$,
  \begin{align}\label{eq:dt-W2-second-curve}
    \frac{\d}{\dt}\frac12 \Wtwometric[\mu][\nu_t][2]
    =
    \int_{\bR^d}
    \inner{w_{\nu_t}\prt{y},y-\Toptmap{\nu_t}{\mu}\prt{y}}
    \d\nu_t\prt{y}.
  \end{align}
  Since $\prt{s,t}\mapsto \Wtwometric[\mu_s][\nu_t][2]$ is absolutely continuous in each
  variable, \cite[Lemma 4.3.4]{AGS2008} gives, for $\mathcal L^1$-a.e.\
  $t\in[0,t_f]$,
  \begin{align*}
    \frac{\d}{\dt}\frac12 \Wtwometric[\mu_t][\nu_t][2]
    \leq
    \frac{\d}{\ds}\frac12 \Wtwometric[\mu_s][\nu_t][2]\bigg|_{s=t}
    +
    \frac{\d}{\ds}\frac12 \Wtwometric[\mu_t][\nu_s][2]\bigg|_{s=t}.
  \end{align*}
  Applying \Cref{eq:dt-W2-first-curve} with $\nu=\nu_t$ and
  \Cref{eq:dt-W2-second-curve} with $\mu=\mu_t$, and changing variables in
  the second integral through the optimal map $\Toptmap{\mu_t}{\nu_t}$, yields
  \Cref{eq:two-curve-ineq}.
\end{proof}

\begin{lemma}[Projected finite-dimensional dynamics is absolutely continuous]
  \label{lemma:finite-dimensional-AC}
  Let $\bW$ be a finite-dimensional
  normed linear space with basis $\{\psi_i\}_{i=1}^m$, continuously embedded in
  $\ambientspace$.
  Let $\bar\dec\in\bW$, and let $\cL:\bW\to\ambientspace$ be continuous.
  Let $\dec_t\in\bW$ solve for $t\in[\bar{t},t_f]$ the
  projected dynamics
  \begin{align*}
    \dot{\dec}_t
    =
    \proj{\bW}{\ambientspace}\prt{\cL\prt{\dec_t}},
    \qquad
    \dec_{\bar t}=\bar\dec,
  \end{align*}
  then $\dec\in AC\prt{[\bar t,t_f];\bW}$.
\end{lemma}

\begin{proof}
  Since $\dec_t\in\bW$, there exist unique coefficients
  $a_t\in\bR^m$ such that
  \begin{align*}
    \dec_t=\sum_{i=1}^m a_{t,i}\psi_i.
  \end{align*}
  Let $G\in\bR^{m\times m}$ be the Gram matrix
  $G_{ij}\coloneqq\inner{\psi_j,\psi_i}_{\ambientspace}$. Since
  $\{\psi_i\}_{i=1}^m$ is a basis of $\bW$, $G$ is invertible. The projected
  dynamics is equivalent to the finite-dimensional ODE
  \begin{align*}
    \dot a_t=F(a_t),
  \end{align*}
  where $F\prt{a}\coloneqq G^{-1}b\prt{a}$, and
  $b_i\prt{a}
    \coloneqq
    \inner[\Big]{
      \cL\prt[\Big]{\sum_{j=1}^m a_j\psi_j},
      \psi_i
    }_{\ambientspace}$.
  By the continuity of $\cL$ on $\bW$, the vector field
  $F:\bR^m\to\bR^m$ is continuous. Since the projected dynamics is solved
  through the induced finite-dimensional ODE, its coordinate curve satisfies
  $\prt{a_t}_{t\in[\bar t,t_f]}\in AC\prt{[\bar t,t_f];\bR^m}$.
  Define the linear coordinate map $A:\bR^m\to\bW$ by
  \begin{align*}
    A \prt{z}\coloneqq \sum_{i=1}^m z_i\psi_i.
  \end{align*}
  Since $\bR^m$ is finite-dimensional, $A$ is bounded as a map into $\bW$.
  Therefore, for every
  $s,t\in[\bar t,t_f]$,
  \begin{align*}
    \norm{\dec_t-\dec_s}_{\bW}
    =
    \norm{A\prt{a_t-a_s}}_{\bW}
    \leq
    \norm{A}_{\Lin\prt{\bR^m;\bW}}\abs{a_t-a_s}_{\bR^m}.
  \end{align*}
  Since $\prt{a_t}_{t\in[\bar t,t_f]}\in AC\prt{[\bar t,t_f];\bR^m}$, the last estimate implies
  $\dec\in AC\prt{[\bar t,t_f];\bW}$.
\end{proof}

\end{document}